\documentclass{amsart}
\usepackage{amssymb,latexsym,url}
\usepackage{blindtext}
\usepackage[a4paper, total={6in, 8in}]{geometry}
\usepackage{xcolor}
\usepackage{enumitem}
\usepackage{cleveref}
\usepackage{graphicx}
\usepackage{todonotes}
\graphicspath{ {./images/} }
\theoremstyle{plain}
\newtheorem{theorem}{Theorem}
\newtheorem*{theorem*}{Theorem}

\newtheorem{proposition}{Proposition}
\newtheorem*{proposition*}{Proposition}
\newtheorem{lemma}{Lemma}
\theoremstyle{definition}
\newtheorem{definition}{Definition}

\newtheorem{notation}{Notation}
\newtheorem{remark}{Remark}

\newtheorem{claim}{Claim}

\DeclareMathOperator{\red}{red}
\DeclareMathOperator{\Res}{Res}

\DeclareMathOperator{\Tr}{Tr}

\DeclareMathOperator{\reg}{reg}

\newcommand{\enm}[1]{\ensuremath{#1}} 

\newcommand{\cal}[1]{\mathcal{#1}}

\newcommand{\NN}{\enm{\mathbb{N}}}

\newcommand{\PP}{\enm{\mathbb{P}}}

\newcommand{\Ii}{\enm{\cal{I}}}

\newcommand{\Ll}{\enm{\cal{L}}}

\newcommand{\Oo}{\enm{\cal{O}}}

\newcommand{\Rr}{\enm{\cal{R}}}

\renewcommand{\phi}{\varphi}
\renewcommand{\theta}{\vartheta}
\renewcommand{\epsilon}{\varepsilon}

\title{Tensoring by a plane maintains secant-regularity in degree at least two
}
\author{E. Ballico, A. Bernardi, T. Ma\'ndziuk}
\address[E. Ballico, A. Bernardi]{Universit\`a di Trento, Via Sommarive, 14 - 38123 Povo (Trento), Italy.}
\address[T. Ma\'ndziuk]{Universit\`a di Trento, Via Sommarive, 14 - 38123 Povo (Trento), Italy, and University of Warsaw, Banacha 2, 02-097 Warsaw, Poland.}
\email{edoardo.ballico@unitn.it,alessandra.bernardi@unitn.it,  tomasz.mandziuk@unitn.it}

\begin{document}

\maketitle

\emph{To our friend and colleague Gianfranco Casnati.}

\begin{abstract}
Starting from an integral projective variety $Y$ equipped with a very ample, non-special and not-secant defective line bundle $\Ll$, the paper establishes, under certain conditions, the regularity of $(Y \times \mathbb P^2,\Ll[t])$ for $t\geq 2$. The mildness of those conditions allow to classify all secant defective cases of  any product of $(\mathbb P^1)^{ j}\times (\mathbb P^2)^{k}$, $j,k \geq 0$, embedded in multidegree at least $(2, \ldots , 2)$ and 
$(\PP^m\times\PP^n\times (\PP^2)^k, \Oo_{\PP^m\times\PP^n\times (\PP^2)^k} (d,e,t_1, \ldots, t_k))$ where $d,e \geq 3$, $t_i\geq 2$, for any $n$ and $m$. 
\end{abstract}

\section*{Introduction}

The problem of classifying the dimensions of secant varieties to algebraic varieties is a classical problem that dates back to the XIX Century with the discovery of the Veronese surface in $\mathbb P^5$  being the only surface in that space whose secant variety of lines was expected to fill the ambient space but it is actually a hypersurface (cf. \cite{Pa}). Then in the 1990s the contributions of F. Zak (\cite{Z}) and J. Alexander and A. Hirschowitz (\cite{ah, ah1, ah2}) made a turning point. In particular the last ones gave the complete classification of all defective secant varieties of Veronese varieties. After that work many other studies have been done with a specific focus on varieties parameterizing tensors (we refer to some of them \cite{CGG2002,cgg1,AOP09,BaurDraisma,bcc,AH11,bbc,cgg,lmr,AMR19,CGG02,BALLICO05,BC06,bcgi07,bcgi, AB11,TV21,CGG05,AOP12,Boralevi201367,BBCC13,BV18,BDHM17}). Nevertheless the unique complete classifications of defectiveness of secant varieties were up to now very few. A. Laface and E. Postinghel classified the products of $\mathbb P^1$'s (cf. \cite{lp}), H. Abo and N. Vannieuwenhoven gave the classification of the tangential varieties of Veronese varieties (cf. \cite{AV}), M.V. Catalisano and A. Oneto covered the tangential varieties of Segre-Veronese varieties (cf. \cite{CO20}) and   F. Galuppi together with A. Oneto classified the Segre-Veronese of two factors embedded in bi-degree at least $(3,3)$ (cf. \cite{go}). See \cite{Hitchhiker} for a comprehensive description of the history of the problem.

Despite its long history and the significant efforts made to classify as many cases as possible, the state of the art is still far from having a complete description of it.

The defectiveness of a secant variety depends on the celebrated Terracini's Lemma (cf. \cite{Terracini}, \cite[Cor. 1.11]{a}), which identifies the relationship between the dimension of the tangent space to the $z$-th secant variety of a variety $X$ and the tangent spaces to $X$ at $z$ general points. A dual form of this lemma is often used, which allows one to transform the geometric problem into an algebraic problem of computing the $h^0$ of the ideal sheaf of $z$ double points twisted with the line bundle that induces the immersion of $X$ (see \Cref{terracini}). This is the technique we also employ in our current work, which enables us to study a broader problem than just focusing on the defectiveness of a single type of variety.

\smallskip

In this work, we tensor a projective and integral variety $Y$ embedded with a not secant defective and non special line bundle $\mathcal{L}$ by a $\mathbb P^2$. We provide sufficient conditions on the dimension of $Y$ and $h^0(Y,\mathcal{L})$ to guarantee the regularity of $(X = Y \times \mathbb{P}^2, \mathcal{L}[t])$ for $t\geq 2$.

Here is the complete statement of our main theorem: 

\begin{theorem*}[Theorem \ref{a5}] Let $Y$ be an integral projective variety of dimension $r$ and $\Ll$ a very ample and not secant defective line bundle on $Y$ such that $h^1(Y,\Ll)=0$.  We denote by  $\alpha:= h^0(Y,\Ll)$ and $X_2:= Y\times \PP^2$.
For the following values of $r$ and $\alpha$
\begin{itemize}
  \item $r=2,3$ and $\alpha \geq 60$,  
        \item $r=4$ and $\alpha \geq 98$, 
        \item $r=5$ and $\alpha \geq 133$, 
        \item $r=6$ and $\alpha \geq 176$, 
        \item $r=7$ and $\alpha \geq 231$, 
        \item $r\geq 8$ and $\alpha \geq \frac{1}{81} (27r^3+144r^2+210r+79)$,
    \end{itemize}
we have that if $t\ge 2$,  then $(X_2,\Ll[t])$ is not secant defective.
\end{theorem*}

The conditions on $r$ and $\alpha$ may seem very restrictive, but in cases where the regularity of secant varieties is guaranteed (either because it is already known or because it can be easily verified  through computation or other means), these conditions allow for an inductive method that provides regularity for entire families of varieties. For example, through the application of our theorem, it becomes evident that Segre-Veronese varieties, constituted by products of $\mathbb{P}^2$'s exclusively and embedded in degrees greater than or equal to 2, are all regular except the  well known cases of the Veronese surface embedded with either $\mathcal{O}(2)$ or $\mathcal{O}(4)$ and the $3\times 3$ matrices embedded with $\mathcal{O}(2,2)$. Here the precise statement.

\begin{theorem*}[\Cref{prod:P2}]
The pair $((\PP^2)^k, \mathcal{O}_{(\PP^2)^k}(t_1,\ldots, t_k))$ with $t_1,\ldots, t_k\geq 2$ is defective if and only if one of the following holds:
\begin{itemize}
    \item $k=1$ and $t_1\in \{2,4\}$,
    \item $k=2$ and $t_1=t_2 = 2$.
\end{itemize}
\end{theorem*}
To our knowledge, this result had not been established previously and it is a complete classification of the defectivity for the product of $\mathbb P^2$'s embedded in any multidegree at least $(2, \ldots ,2)$. 

\smallskip  

We also generalize the above classification to any product of $\mathbb P^1$'s and $\mathbb P^2$'s embedded in multidegree at least $(2, \ldots ,2)$ (cf. \Cref{thm:p1andp2}).

\smallskip

Expanding on our main result's application, we demonstrate another significant implication.
\begin{theorem*}[\Cref{thm:mn2}]
If $d,e \geq 3$ and $n\geq m\geq 1$, then the pair 
\[
(\PP^m\times\PP^n\times (\PP^2)^k, \Oo_{\PP^m\times\PP^n\times (\PP^2)^k} (d,e,t_1, \ldots, t_k))
\]
is not secant defective for all $k$ and $t_1\geq t_2\geq \cdots \geq 2$.
\end{theorem*}

\smallskip

The main technique used in our work is well-known as the method of \emph{Differential Horace}, introduced by Alexander and Hirschowitz in 1992 (cf. \cite{ah1, ah}). The significant  technical innovation we introduce is the simultaneous application of this method to a divisor $H$ and a divisors of $H$ itself within the same framework  (cf. \Cref{rmk:simultaneous_Horace}).

\smallskip

The structure of the paper is the following.

In the Preliminary \Cref{Preliminaries} we recall the Horace Differential Lemma and how to apply it in our specific instances. We explain its simultaneous version and then we recall how a semicontinuity argument allows to reduce the proof of the regularity to a specialization of points.

In \Cref{MainResult} we collect the lemmas that lead to the main theorem (\Cref{a5}) which are the first steps of the inductive procedure. Some of those lemmas are technical and need computational claims whose proofs are collected in the final \Cref{Section:claims}.

In \Cref{Applications}, we present straightforward examples that demonstrate the broad applicability of our theorem. These instances play a crucial role in finalizing the classification of defective cases, resolving numerous open cases comprehensively. In particular we prove the already mentioned \Cref{prod:P2}, \Cref{thm:p1andp2} and \Cref{thm:mn2}.

\subsection*{Acknowledgments} EB and AB are member of GNSAGA of INDAM.
AB and TM have been partially supported by the Thematic Research Programme ``Tensors: geometry, complexity and quantum entanglement", University of Warsaw, Excellence Initiative – Research University and the Simons Foundation Award No. 663281.
All authors are members of TensorDec laboratory of the Mathematical Department of Trento.

Funded by the European Union under NextGenerationEU. PRIN 2022 Prot. n. 2022ZRRL4C$\_$004. Views and opinions expressed are however those of the authors only and do not necessarily reflect those of the European Union or European Commission.  Neither the European Union nor the granting authority can be held responsible for them.
\\
\includegraphics[scale=0.5]{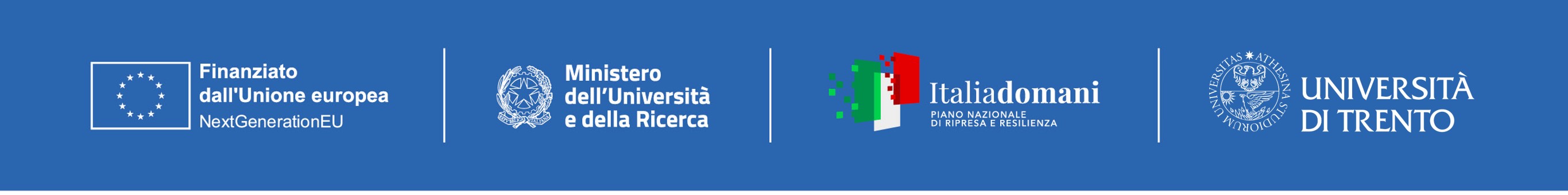} 

\section{Preliminaries}\label{Preliminaries}

\begin{notation}
Let $X$ be a projective variety. For any non singular point $p\in X_{\reg}$ of $X$, with  $(2p,X)$ we denote the closed subscheme of $X$ with $(\Ii_{p,X})^2$ as its ideal sheaf. The degree of this scheme is $\deg ((2p,X)) =\dim X+1$ and its support is $(2p,X)_{\red} =\{p\}$. For any finite set $S\subset X_{\reg}$, set $(2S,X):= \cup_{p\in S} (2p,X)$. We often write $2p$ and $2S$ instead of $(2p,X)$ and $(2S,X)$ if the ambient variety is obvious from the context.
\end{notation}

\begin{remark}[Terracini Lemma]\label{terracini}
For any positive integer $z$, denote the $z$-th secant variety of a projective variety $X$ as $\sigma _z(X)$. The celebrated  Terracini Lemma (cf. \cite{Terracini}, \cite[Cor. 1.11]{a}, \cite[5.3.1.1]{l}) can be stated in the following form: The dimension $\dim \sigma _z(X)$ is the dimension of the linear span of the zero-dimensional scheme $(2S,X)$, where $S$ is a general subset of $X$ with $\#S=z$. As a consequence, if the embedding of $X$ is induced by the complete linear system $|\mathcal{L}|$, then $\dim \sigma _z(X) =h^0(X,\mathcal{L})-1-h^0(X,\Ii _{(2S,X)}\otimes \mathcal{L})$. Hence $\dim \sigma _z(X) = z(\dim X+1)-1$ if and only if $h^1(X,\Ii _{(2S,X)}\otimes \mathcal{L}) =h^1(X,\mathcal{L})$. 
\end{remark}
\begin{definition}
    Let  $\mathcal{L}$ be a line bundle on a projective variety $X$  with $h^1(X,\mathcal{L})=0$. The pair $(X,\mathcal{L})$  is said to be \emph{not-$z$-secant defective} if for a general union $Z$ of $z$ double points of $X$ the following equality holds $$h^0(X,\mathcal{I}_Z\otimes \mathcal{L})\cdot h^1(X,\mathcal{I}_Z\otimes \mathcal{L})=0.$$ Should $(X,\mathcal{L})$ consistently be not-$z$-secant defective for all $z\geq 1$, it is referred to as \emph{not being secant defective}.
\end{definition}

Remark that $h^0(X,\mathcal{I}_Z\otimes \mathcal{L})=0$ is equivalent to the fact that $\sigma_z(X)$ fills the ambient space; while $h^1(X,\mathcal{I}_Z\otimes \mathcal{L})=0$ corresponds to $\dim(\sigma_z(X))=z(\dim X +1)-1$.

\subsection{Horace Differential Lemma}
One of the main tools that we will use in this paper is the Horace Differential Lemma (cf. \cite{ah1, ah}) and its simultaneous version in codimensions 1 and 2. We present here the versions that we will use.

We first recall few basic definitions.

Let $X$ be an {integral} projective variety and $H$ an { integral} effective Cartier divisor of $X$. For any zero-dimensional scheme $W\subset X$ let $\Res_H(W)$ denote the residual scheme of $W$ with respect to $H$, i.e. the closed subscheme of $X$ with $\Ii_W:\Ii_H$ as its ideal sheaf.
We have $\deg (W) =\deg (\Res_H(W)) +\deg (W\cap H)$, $\Res_H(W)\subseteq W$, $\Res_H(W) =W$ if $W\cap H=\emptyset$, $\Res_H(W) =\emptyset$ if $W\subset H$ and $\Res_H(W) =\Res_H(A)\cup \Res_H(B)$ if $W=A\cup B$ and $A\cap B=\emptyset$. If $p\in H_{\reg}\cap X_{\reg}$, then $\Res_H((2p,X)) = \{p\}$ and $(2p,X)\cap H =(2p,H)$. For any line bundle $\Rr$ on $X$ there is an exact sequence
\begin{equation}\label{eqp1}
0 \to \Ii_{\Res_H(W)}\otimes \Rr(-H)\to \Ii_W\otimes \Rr \to \Ii_{W\cap H,H}\otimes \Rr_{|H} \to 0
\end{equation}
of coherent sheaves on $X$ which we call the {\it residual sequence of $H$}. The scheme $W\cap H$ is often called the trace of $W$ on $H$ and denoted by $\Tr_H(W)$.

The following definition is originally from \cite{ah} and we quote it as stated in \cite{bcgi}.
\begin{definition}\label{def:vertically:graded}
In the algebra of formal functions $\kappa
[[{\bf x},y]]$, where ${\bf x} = (x_1,\ldots,x_{n-1})$, a {\it
vertically graded} (with respect to $y$) ideal is an ideal of the
form:
$$
I = I_0 \oplus I_1y \oplus \cdots \oplus I_{m-1}y^{m-1}\oplus (y^m)
$$
where for $i = 0,\ldots,m-1$, $I_i\subset \kappa [[{\bf x}]]$ is an
ideal.
\par
Let $X$ be an $n$-dimensional projective variety, let $H$ be an integral divisor on $X$. We say that $Z \subset X$ is a
{\it vertically graded subscheme} of $X$ with base $H$ and support
$z\in H$, if $Z$ is a $0$-dimensional scheme with support at a regular point $z$ of $X$ such that there is a regular system of parameters  $({\bf x},y)$
at  $z$ such that $y=0$ is a local equation for $H$ and the ideal of
$Z$ in $\widehat {\cal O}_{X,z} \cong \kappa [[{\bf x},y]]$ is
vertically graded.

 Let $Z\subset X$ be a vertically graded subscheme with base $H$, and $p\geq 0$
be a fixed integer; we denote by  $\Res^{p}_H(Z) \subset X$ and
$\Tr^{p}_H(Z)\subset H$  the closed subschemes defined, respectively, by
the ideals:
$$ {\cal I}_{\Res^{p}_H(Z)} := {\cal I}_{Z}+
({\cal I}_{Z}:{\cal I}^{p+1}_H){\cal I}^{p}_H, \qquad \qquad \qquad
{\cal I}_{\Tr^{p}_H(Z),H} := ({\cal I}_{Z}:{\cal I}^{p}_H)\otimes
{\cal O}_H
$$
and we call them the \emph{$p$-th virtual residue (trace) of $Z$ with respect to $H$} (respectively).
Notice
that for $p=0$ we get the usual trace $\Tr_H(Z)$ and residual schemes $\Res_H(Z)$.

Let $Z_1,\ldots ,Z_r\subset X$ be vertically graded subschemes with base $H$
and supports $z_i \in H$, $i=1,\ldots , r$ pairwise distinct,  $Z=Z_1\cup \cdots \cup Z_r$, and ${\bf
p}=(p_1,\ldots ,p_r)\in {\mathbb N}^r$.
\par We set:
$$\Tr^{\bf p}_H(Z):= \Tr^{p_1}_H(Z_1)\cup \cdots  \cup \Tr^{p_r}_H(Z_r)$$
$$\Res^{\bf p}_H(Z):= \Res^{p_1}_H(Z_1)\cup \cdots \cup
\Res^{p_r}_H(Z_r).$$
\end{definition}

\begin{lemma}[Horace Differential Lemma \cite{ah}]
\label{diff:H:lemma:vertically}
Let $X$ be {an integral} projective variety and $H$ be an integral {effective Cartier} divisor. Let $\mathcal{R}$ be a very ample line bundle. Let $W\subset X$ be a $0$-dimensional scheme supported in the regular locus of $X$. 
Let $A_1,\ldots ,A_r,\, B_1,\ldots ,B_r$ be $0-$dimensional irreducible
subschemes of $X$ such that $A_i\cong B_i$, $i=1,\ldots ,r$, $B_i$
has support on $H$ and is vertically graded with base $H$, and the
supports of $A=A_1\cup \cdots \cup A_r$ and $B=B_1\cup \cdots \cup B_r$ are general in their respective Hilbert schemes. Let ${\bf
p}=(p_1,\ldots ,p_r)\in {\mathbb N}^r$. Then for $i=0,1$:
$$ h^i(X,{\cal I}_{W\cup A}\,\otimes \mathcal{R})\leq h^i(X,{\cal I}_{\Res_HW\cup \Res^{\bf p}_H(B)}\otimes \mathcal{R}(-H))+h^i(H,{\cal I}_{\Tr_HW\cup \Tr^{\bf p}_H(B),H}\otimes \mathcal{R}|_H). $$
\end{lemma}

\begin{lemma}[Horace Differential Lemma for double points]\label{Differential Horace lemma}
Let $X$ be an {integral} projective variety and $H\subset X$ be an integral {effective Cartier} divisor. Let $\mathcal{R}$ be a very ample line bundle. Fix  
an integer $r>0$.
Let $W\subset X$ be a zero-dimensional scheme supported in the regular locus of $X$, $\widetilde{S}\subset X$ be a general subset of $H$ with $\#\widetilde{S}=r$ and $\widetilde{Z}\subset X$ be a general set with $\#\widetilde{Z} = r$. Then for $i\in \{0,1\}$ we have
$$h^i(X, \Ii_{W\cup (2\widetilde{Z},X)}\otimes \mathcal{R}) \leq h^i(X,\Ii_{\Res_HW\cup (2\widetilde{S},H)}\otimes \mathcal{R}(-H))+ h^i(H,\Ii _{(\Tr_HW)\cup \widetilde{S}}\otimes \mathcal{R}_{|H}).
$$
\begin{proof}
    We apply \Cref{diff:H:lemma:vertically} with $A=(2\widetilde{Z},X)$, $B= (2\widetilde{S},X)$ and $\mathbf{p}=(1, \ldots , 1)$.
\end{proof}
\end{lemma}

The following \Cref{rmk:simultaneous_Horace} is one of our main strategy since it allows to apply Horace Differential Lemma  simultaneously to a divisor $H'$ of $X$ and a divisor $E$ of $H'$ itself.

\begin{lemma}[Simultaneous Horace Differential Lemma]\label{rmk:simultaneous_Horace}
Let $X$ be an integral projective variety with a very ample line bundle $\mathcal{R}$.
Take integral {effective Cartier} divisors $H$ and $H'$ such that $H\ne H'$ and $E:=H\cap H'$ is integral.  Let $\widetilde{S}$ (resp. $\widetilde{Z}$) be a set of $r$ general points of $E$ (resp. $H$).
Let $W$ be a zero-dimensional subscheme of $X$ contained in the regular locus of $X$ and such that $W \cap \widetilde{S} = W \cap \widetilde{Z} = \emptyset$. 
For $i=0,1$ we have
$$h^i(X,\Ii_{W\cup (2 \widetilde{Z},H)}\otimes \Rr) \le h^i(X,\Ii_{\operatorname{Res}_{H'}(W) \cup (2 \widetilde{S},E)}\otimes \mathcal{R}(-H')) +h^i(H', \Ii_{(\operatorname{Tr}_{H'}(W))
\cup \widetilde{S}}\otimes \mathcal{R}_{|H'}).$$
\end{lemma}

\begin{figure}[h]
\includegraphics[width=10cm]{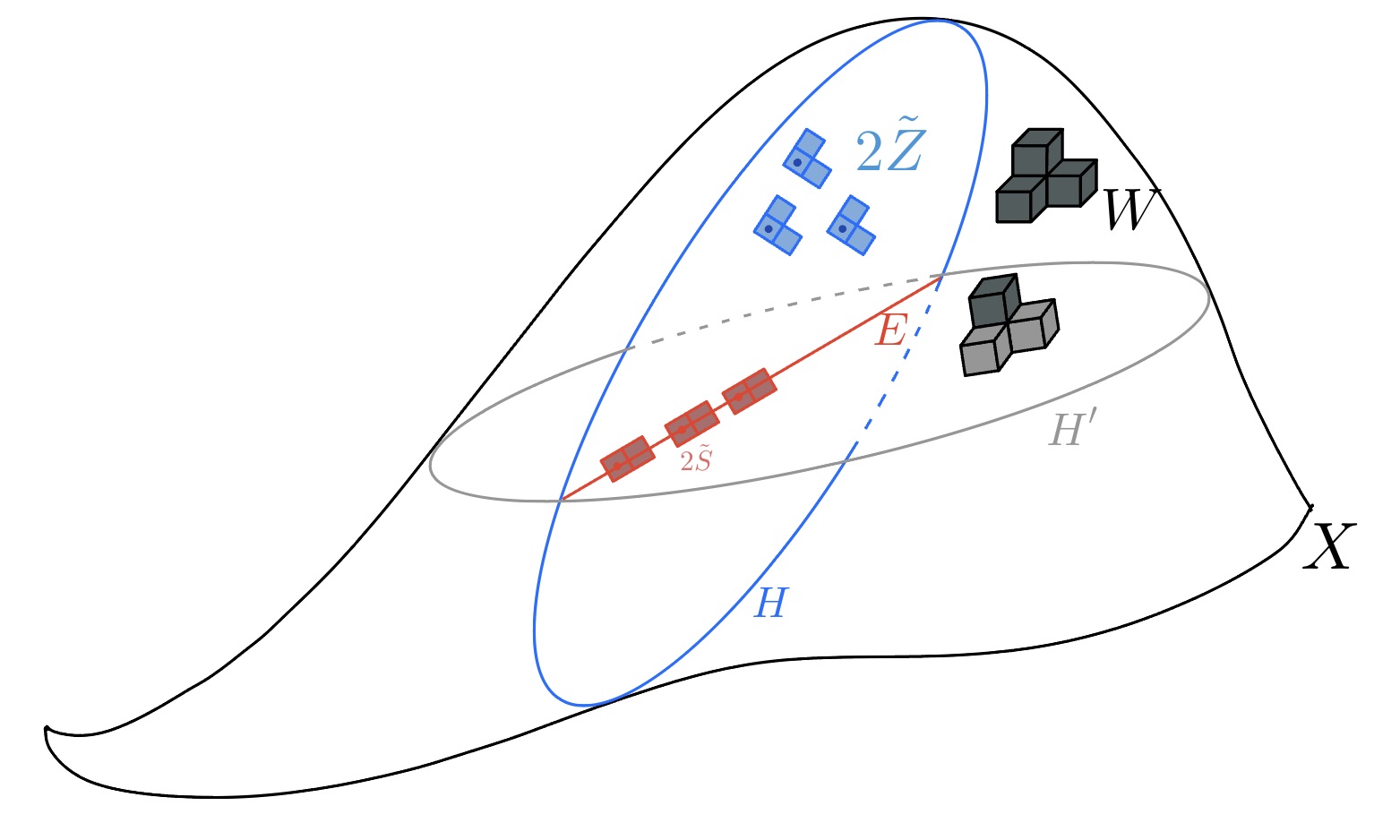}
\caption{{\footnotesize{This is an intuitive graphical representation of \Cref{rmk:simultaneous_Horace} in a special case where $\dim X=3$, $\dim H=\dim H'=2$ so $\deg (2\widetilde Z)= 3 (\# \widetilde Z)$  and $\dim E=1$ so $\deg (2 \widetilde S)= 2 (\# \widetilde S)$. The residual of $W$ mod $H'$ is in dark gray, while the trace in light gray. The dots denote the supports of the corresponding schemes.}}}
\end{figure}

\begin{proof}
The main idea is to apply \Cref{diff:H:lemma:vertically} to $H=H'$,  $A = (2\widetilde{Z}, H)$, $B = (2 \widetilde{S}, H)$ and $\mathbf{p} = (1,\ldots,1)$, but let us give more details.

Set $n:= \dim X$. We degenerate each point of $\widetilde{Z}$ to a different point of $\widetilde{S}$. For each $p\in \widetilde{Z}$ call $p'$ the corresponding point of $\widetilde{S}$. In the degeneration the zero-dimensional scheme $(2 p,H)$ would have as a limit the scheme $(2p',H)$. Let $\kappa[[x_1,\dots ,x_n]]$ denote the formal completion of the local ring $\Oo_{X,p'}$ with $x_n=0$ a local equation of $H$ (not of $H'$) at $p'$. The scheme $(2p',H)$ is vertically graded in the sense of \cite[p. 392]{ah}, i.e. locally one can write the ideal of $(2p',H)$ as $\bigoplus_{i= 0}^mI_ix_{n-1}^i$ with $x_{n-1}$ the local equation of $H'$ and $I_i\subset \mathbb{K}[[x_1, \ldots , x_{n-2},x_n]]$. The virtual trace of $(2p',H)$ with respect to $H'$ is $p'$, while
the virtual residue is $(2p',E)$, because $H\cap H'=E$. Small explanation: if we started with $(2p',X)$ for the usual  Horace Differential Lemma the usual virtual trace would $p'$ and virtual residue $(2p',H')$.
Here for the smaller scheme we get the same virtual trace and exactly $(2p',E)$ as virtual residue, because $p$ goes to $p'$ remaining in $H$ and $E=H\cap H'$.
\end{proof}

\subsection{Specializing and adding simple points}

\begin{notation}\label{NotationYandX}
From now on we fix the following:
\begin{itemize}
    \item $Y$ will be an integral projective variety of dimension $r$,

    \item $\Ll$ a very ample line bundle on $Y$ such that $h^1(Y,\Ll)=0$,
     and $\alpha:= h^0(Y,\Ll)$,
    \item $X_m:= Y\times \PP^m$ with $m\geq 1$,
    \item { $H := Y\times \PP^{m-1}$ for some hyperplane $\PP^{m-1}\subseteq \PP^m$}.
\end{itemize}
With these assumptions we have
$h^0(X_m,\Ll[t]) =\alpha \binom{m+t}{m}$ and $h^1(X_m,\Ll[t]) =0$ for all $t\in \NN$.
\end{notation}

\begin{remark}\label{lol1}
Let $f: U_1\to U_2$ be a surjective submersion of smooth and connected quasi-projective varieties.
Fix $p\in U_1$. Let $A_1$ (resp. $A_2$) be the completion of the local ring $\Oo_{U_1,p}$ (resp. $\Oo_{U_2,f(p)}$).
Since $U_1$ is smooth at $p$ and $U_2$ is smooth at $f(p)$, $A_i$ is a ring of power series in $\dim A_i$ variables. 
Set $n_2:=\dim A_2$ and $n_1:= \dim A_1$. Since $f$ is a submersion, we have formal coordinates $x_1,\dots ,x_{n_1}$ of $A_1$
such that the variables $x_1,\dots ,x_{n_2}$ are formal variables for the power series ring $A_2$. Thus { the scheme $2f(p)\subset U_2$ is the scheme-theoretic image of $f(2p)$.}
If $S\subset U_1$ is a general finite set, then $f_{|S}$ is injective and $f(S)$ is a general subset of $U_2$ with cardinality $\#S$.
Thus if $Z\subset U_1$ is a general union of $z$ double points of $U_1$, then $f(Z)$ is a general union of $z$ double points of $U_2$.

Let $X_m$ and $Y$ be as in \Cref{NotationYandX}. Consider $\pi_1$ and $\pi_2$ as the projections from $X_m$ onto its first and second factors, respectively. Let $U_1:= Y_{\reg}$, $U_2:= \pi_1^{-1}(Y_{\reg})=X_{m,\reg}$ and $\pi := \pi_{1|U_2}$.
If $Z\subset X_m$ is a general union of $z$ double points of $X_m$ and $u$ double points of $H$, then $\pi_1(Z)$ is a general union of  $z+u$ double points of $Y$.
Recall that $\Ll[0]:= \pi_1^\ast (\Ll)$ and hence $\pi _1^\ast$ induces an isomorphism $H^0(Y,\Ll)\cong H^0(X_m,\Ll[0])$. Thus $H^0(X_m,\Ii_Z\otimes \Ll[0])\cong 
H^0(Y,\Ii_{f(Z)}\otimes \Ll)$. If $(Y, \mathcal{L})$ is not-$(z+u)$-secant defective, then $h^0(Y,\Ii_{f(Z)}\otimes \Ll) = \max \{0,\alpha -(r+1)(z{+u})\}$.
\end{remark}

\begin{remark}[Semicontinuity]\label{specialization}  Let $X$ be an integral projective variety, $H\subset X$ be an integral effective Cartier divisor of $X$, $\Rr$ a line bundle on $X$ and $Z_X\subset X$
a zero-dimensional scheme on $X$. Let $Z_H$ be a specialization of $Z_X$ on $H$. By the semicontinuity theorem for cohomology (cf. \cite[Ch. III]{Har}), if one can prove that $h^i(X,\mathcal{I}_{Z_H}\otimes \mathcal{R})=0$
for special scheme $Z_H$ then also $h^i(X,\mathcal{I}_{Z_X}\otimes \mathcal{R})=0$ for the general scheme $Z_X$.
\end{remark}

\begin{lemma}\label{lo2}
Let $X$ be an integral projective variety such that $\dim X\ge 2$, $H\subset X$ be an integral effective Cartier divisor of $X$, $\Rr$ a line bundle on $X$
and $W\subset X$ a zero-dimensional scheme. Set $\beta:= h^0(X, \Ii_W\otimes \Rr)$. Fix an integer $e\geq 0$ and let $S\subset H$ be a general subset of $H$ with cardinality $e$.
Then $h^0(X,\Ii_{W\cup S}\otimes \Rr) \le \max \{\beta -e,h^0(X,\Ii_{\Res_H(W)}\otimes \Rr(-H))\}$.
\end{lemma}

\begin{proof}
The lemma is trivial if $e=0$ and hence we may assume $e>0$ and use induction on the integer $e$. We order the points $p_1,\dots ,p_e$ of $S$ and set $S':= S\setminus \{p_e\}$.
By the inductive assumption we have $h^0(X,\Ii_{W\cup S'}\otimes \Rr) \le \max \{\beta -e+1,h^0(X,\Ii_{\Res_H(W)}\otimes \Rr(-H))\}$. Since $h^0(X,\Ii_{W\cup S}\otimes \Rr) \le h^0(X,\Ii_{W\cup S'}\otimes \Rr)$,
we get the lemma if $h^0(X,\Ii_{W\cup S'}\otimes \Rr)\leq h^0(X,\Ii_{\Res_H(W)}\otimes \Rr(-H))$. Now assume $h^0(X,\Ii_{W\cup S'}\otimes \Rr)=\beta -e+1>h^0(X,\Ii_{\Res_H(W)}\otimes \Rr(-H))$. 
Since the lemma is true if $h^0(X,\Ii_{W\cup S}\otimes \Rr) <h^0(X,\Ii_{W\cup S'}\otimes \Rr)$, we may assume $h^0(X,\Ii_{W\cup S}\otimes \Rr) =h^0(X,\Ii_{W\cup S'}\otimes \Rr)$. Note that $\Res_H(W\cup S') =\Res_H(W)$.
Since we take $p_e$
general in $H$ after fixing $W\cup S'$, $H$ is contained in the base locus of $|\Ii_{W\cup S'}\otimes \Rr|$. Thus $h^0(X,\Ii_{W\cup S'}\otimes \Rr)=h^0(X,\Ii_{\Res_H(W)}\otimes \Rr(-H))$.
\end{proof}

\begin{remark}\label{lo3}
In the content of \Cref{lo2}, if we assume that $h^1(X,\Ii_W\otimes \Rr)=0$, then the following are true.
\begin{itemize}
\item If $h^0(X,\Ii_{\Res_H(W)}\otimes \Rr(-H))\le \beta -e$, we get $h^0(X,\Ii_{W\cup S}\otimes \Rr) =\beta -e$
and hence $h^1(X,\Ii_{W\cup S}\otimes \Rr) =0$.
\item If $h^0(X,\Ii_{\Res_H(W)}\otimes \Rr(-H))=0$ and $\beta <  e$, we get $h^0(X,\Ii_{W\cup S}\otimes \Rr) =0$.
\end{itemize}
\end{remark}

\section{Main result}\label{MainResult}
We remind that, as settled in \Cref{NotationYandX},  $Y$ is an integral projective variety of dimension $r$ embedded with a very ample line bundle $\Ll$ such that $h^1(Y, \Ll) = 0$, $\alpha:= h^0(Y,\Ll)$ and
 $X_m:= Y\times \PP^m$ with $m\geq 1$.

Establishing the regularity of $(X_2, \mathcal{L}[t])$ relies  on an inductive approach, requiring a foundational understanding of essential results concerning $X_1$.

\begin{lemma}[{\cite[Theorem 2]{b}}]\label{a4.1}
 Let $X_1= Y \times \mathbb P^1$, with $(Y,\mathcal{L})$ and $\alpha$ as  in \Cref{NotationYandX};  $\dim(Y)=r>1$, $\alpha >(\dim Y +1)^2$.  { If the pair $(Y, \Ll)$ is not $\lfloor \alpha/(r+1)\rfloor$-secant defective,} then the pair $(X_1,\Ll[t])$ is not secant defective for all $t\ge 2$.
\end{lemma}

\begin{lemma}\label{a1} Let $X_1= Y \times \mathbb P^1$, with $(Y,\mathcal{L})$, $\alpha$ as in \Cref{NotationYandX} 
and  $\dim(Y)=r>1$.  If $Z\subset X_1$ is a general union of $z$ double points, then

\begin{enumerate}[label=(\alph*)]
    \item\label{firsta} 
    $h^1(X_1,\Ii_Z\otimes \Ll[1]) =0$ if $z \le 2\lfloor \alpha/(r+2)\rfloor$ { and $(Y, \Ll)$ is not $\lfloor \alpha/(r+2) \rfloor$-secant defective}. 
    \item\label{firstb}  $h^1(X_1,\Ii_Z\otimes \Ll[1]) =0$ if $(r+2)z \le 2\alpha -2r-2$ { and $(Y, \Ll)$ is not $\lfloor \alpha/(r+2) \rfloor$-secant defective}. 
    \item \label{firstc}  $h^0(X_1,\Ii_Z\otimes \Ll[1]) =0$ if $z \ge 2\lceil \alpha/(r+2)\rceil$ { and $(Y, \Ll)$ is not $\lceil \alpha/(r+2) \rceil$-secant defective}. 
    \item \label {firstd} $h^0(X_1,\Ii_Z\otimes \Ll[1]) =0$ if $(r+2)z \ge 2\alpha +2r+2$ { and $(Y, \Ll)$ is not $\lceil \alpha/(r+2) \rceil$-secant defective}. 
\end{enumerate}

\end{lemma}

\begin{proof}
Note that \cref{firsta} implies \cref{firstb}. To prove \cref{firsta} we may assume $z = 2\lfloor \alpha/(r+2)\rfloor$.
Fix $o\in \PP^1$ and set $H:= Y\times \{o\}\in |\Oo_{X_1}[1]|$. Take a general $\overline{S}\cup \widetilde{S}\subset H$ such that $\#\overline{S} =\#\widetilde{S}= \lfloor \alpha/(r+2)\rfloor$. Decompose the support of $Z$ into the disjoint union of two sets $\overline{Z} \cup \widetilde{Z}$ where $\#\overline{Z} = \# \widetilde{Z}  =\lfloor \alpha/(r+2)\rfloor$. In order to show that $h^1(X_1,\Ii_Z\otimes \Ll[1]) = 0$ it is sufficient by semicontinuity to show that
$h^1(X_1,\Ii_{(2\overline{S},X_1) \cup (2\widetilde{Z},X_1)} \otimes \Ll[1]) = 0$. 

Note that $\deg(\overline{S} \cup (2\widetilde{S},H)) = \deg((2\overline{S},H) \cup \widetilde{S})  = (r+2)\lfloor \alpha/(r+2)\rfloor \le \alpha$. Since {$(Y,\Ll)$ is not $\lfloor \alpha/(r+2)\rfloor$-secant defective} and $\overline{S}\cup \widetilde{S} \subset H$ {is general} we have
\[
h^1(X_1,\Ii_{\overline{S} \cup (2\widetilde{S},H)}\otimes \Ll[0]) = h^1(H,\Ii_{(2\overline{S},H) \cup \widetilde{S}}\otimes \Ll[1]_{|H}) = 0.  
\]
Therefore, by Horace Differential Lemma for double points (\Cref{Differential Horace lemma}) applied with $W = (2\overline{S},X_1)$ and $\Rr = \Ll[1]$ we get $h^1(X_1,\Ii_{(2\overline{S},X_1) \cup (2\widetilde{Z},X_1)} \otimes \Ll[1]) = 0$ which finishes the proof of \cref{firsta}.

The proof of \cref{firstc} and hence of \cref{firstd} is similar reversing the inequalities.  
\end{proof}

\begin{lemma}\label{a3}
Let $X_2=Y \times \mathbb{P}^2$ with $(Y, \Ll)$, 
 $r$ and $\alpha$ as in \Cref{NotationYandX}. Let $L$ be a line in $\PP^2$ and $Z \subset X_2$ be a general union of $z$ double points of $X_2$ and $u$ double points of $H:= Y \times { L}$. 
\begin{enumerate}[label=(\alph*)]
    \item\label{seconda} $h^1(X_2,\Ii_Z\otimes \Ll[1]) =0$ if $(r+2)z\le 2\alpha - 2r-2$, $(r+2)z+u \leq 2\alpha$,  $u\le \lfloor \alpha/(r+1)\rfloor$, $z\le \alpha -(r+1)u$ { and $(Y, \Ll)$ is not-$s$-secant defective for all $s\in \{z,u, \lfloor \alpha/(r+2)\rfloor\}$}.
    \item \label{secondb}$h^0(X_2,\Ii_Z\otimes \Ll[1]) =0$ if { $(Y, \Ll)$ is not-$s$-secant defective for  $s\in \{u, \lceil \alpha/(r+2)\rceil\}$, } $(r+2)z\ge 2\alpha + 2r + 2$, 
    and either $u\ge \lceil \alpha/(r+1)\rceil$ or $z\geq \alpha - (r+1)u.$
\end{enumerate}

\end{lemma}
\begin{proof}
Fix a line $L'$ in $\PP^2$ different from $L$ and let $o$ be the point of intersection of these lines. Call $E:=Y \times \{o\}$.
Let $\widetilde{Z}$ be the union of the supports of the $u$ double points of $H$ that are connected components of $Z$. Fix a general $\widetilde{S}\subset E$ such that $\#\widetilde{S}=u$ and a general $\overline{S}\subset H'$
such that $\#\overline{S}=z$. We specialize the $z$ double points of $Z$ to $(2\overline{S},X_2)$. As recalled in \Cref{specialization} if we can prove that such a specialization is regular then we will have that also $Z$ is regular. So, in order to prove \cref{seconda} now we need to show that $h^1(X_2,\mathcal{I}_{(2\overline{S},X_2) \cup (2\widetilde{Z},H)}\otimes \mathcal{L}[1])=0$ in the right range. 
It follows from \Cref{rmk:simultaneous_Horace} applied with $W = (2\overline{S},X_2)$ and $\mathcal{R} = \Ll[1]$ that it is sufficient to prove that $h^1(H',\Ii_{(2\overline{S},H')\cup \widetilde{S}}\otimes \Ll[1]_{|H'}) = 0$ and $h^1(X_2,\Ii_{\overline{S} \cup (2\widetilde{S},E) }\otimes \Ll[0])=0$.

{\hangindent=1cm

\begin{claim}\label{claim6}
   $h^1(H',\Ii_{(2\overline{S},H')\cup \widetilde{S}}\otimes \Ll[1]_{|H'}) = 0$.
\end{claim}
\begin{proof}[Proof of \Cref{claim6}]
By \Cref{a1} we have $h^1(H',\Ii_{(2\overline{S},H')}\otimes \Ll[1]_{|H'})=0$. Therefore, 
\[
h^0(H', \mathcal{I}_{(2\overline{S}, H')}\otimes \Ll[1]_{|H'}) = 2\alpha - (r+2)z \geq u.
\]
We have $h^0(H',\Ii_{\mathrm{Res}_{E}(2\overline{S}, H')}\otimes \Ll[0]_{|H'})=h^0(H', \mathcal{I}_{(2\overline{S}, H')} \otimes \Ll[0]_{|H'})$ which is equal to $\max \{0, \alpha - (r+1)z\}$ by \Cref{lol1}. 
By \Cref{lo3} in order to prove the claim it is sufficient to show that
\[
\alpha - (r+1)z \leq 2\alpha - (r+2)z  - u.
\]    
This is equivalent to $z+u \leq \alpha$ which follows from the inequalities in the statement of this \Cref{a3}.
\end{proof}
}

Now it is sufficient to show that $h^1(X_2,\Ii_{\overline{S}\cup (2\widetilde{S},E)}\otimes \Ll[0])=0$.
By \Cref{lol1} we have $h^1(X_2,\Ii_{(2\widetilde{S},E)}\otimes \Ll[0]) =0$, because $(Y,\Ll)$ is not-$u$-secant defective and $u\le \lfloor \alpha/(r+1) \rfloor$. Note that $h^0(X_2,\Ii_{(2\widetilde{S},E)}\otimes \Ll[0]) = \alpha -(r+1)u\ge z$. Since $\overline{S}$ is general in $H'$ and $h^0(X_2,\Ii_{(2\widetilde{S},E)}\otimes \Ll[0]) \ge z$, by \Cref{lo3} to prove  \cref{seconda} it is sufficient to use that $h^0(X_2,\Ll[-1])=0$.

The proof of \cref{secondb} is similar and omitted. \end{proof}

\begin{lemma}\label{a4old} Let $X_2=Y\times \mathbb P^2$ where $(Y, \mathcal{L})$, $r$ and $\alpha$ are as in \Cref{NotationYandX}. { If the pair $(Y, \Ll)$ is not secant defective, then} for all the following values of $r$ and $\alpha$ we have that the pair $(X_2,\mathcal{L}[2])$ is not secant defective.    \begin{itemize}
        \item $r=2$ and $\alpha \geq 71$, 
        \item $r=3$ and $\alpha \geq 75$, 
        \item $r=4$ and $\alpha \geq 99$, 
        \item $r=5$ and $\alpha \geq 138$, 
        \item $r=6$ and $\alpha \geq 183$, 
        \item $r=7$ and $\alpha \geq 234$, 
        \item $r\geq 8$ and $\alpha \geq \frac{1}{81} (27r^3+144r^2+210r+79)$. 
    \end{itemize}
\end{lemma}
\begin{proof} Let $Z\subset X_2$ be a general union of $z$ double points
with  $z\in \{\lfloor 6\alpha/(r+3)\rfloor, \lceil 6\alpha/(r+3)\rceil\}$. Set $j=1$ if $z = \lfloor 6\alpha/(r+3)\rfloor$ and $j=0$, otherwise.   In order to obtain that $(X_2,\mathcal{L}[2])$ is not secant defective it is sufficient to show that $h^j(X_2,\mathcal{I}_Z\otimes \mathcal{L}[2]) = 0$.
We will obtain this vanishing of cohomology by applying the Horace Differential Lemma to a suitable specialization of $Z$ that we now describe.
First of all consider a divisor  $H=Y\times \mathbb P^1$ of $X_2$ as in \Cref{NotationYandX}. Write $Z= Z' \cup (2 \overline{Z},X_2) \cup (2 \widetilde{Z},X_2)$ with $\#\overline{Z}=x_1$ and $\#\widetilde{Z}=y_1$ for some integers $x_1,y_1$ (we first finish to explain the construction and we will specify  the precise values of $x_1$ and $y_1$ later in \Cref{claim1}). Then  specialize $\overline Z$ on $H$ and  call it $\overline S$. Consider $\widetilde S \subset H$ a set of general $y_1$ points disjoint from $\overline S$. Now, in order to apply the Horace Differential Lemma for double points (\Cref{Differential Horace lemma}) we need $x_1$ and $y_1$ to be properly chosen.

{\hangindent=1cm
\begin{claim}\label{claim1}
    There exist non-negative integers $x_1$ and $y_1$ such that $(r+2)x_1+y_1 = 3\alpha$, $x_1+y_1 \leq z$,  $2r+2 \leq y_1\leq \lfloor \alpha/(r+1) \rfloor$, $(r+2)(z-x_1-y_1) + y_1\leq 2\alpha$, $z-x_1-y_1 \leq \alpha - (r+1)y_1$.
\end{claim}
}
The proof of this claim is computational only and it is collected in \Cref{Section:claims}. However it is worth to remark that the statement of \Cref{claim1} is a consequence of a common solutions for the inequalities \cref{eqcon1}, \cref{eqcon2}, \cref{eqcon3} and \cref{eqcon4}. 
Remark that for $r=2,3$, \Cref{claim1} holds also for values of $\alpha$ smaller than $71$ and $75$ (resp.) of the statement of the present Lemma, but the inequality~\cref{eq8}  used in the proof of \Cref{claim3} holds only for $\alpha\geq71,75$ (resp.). 
By semicontinuity, in order to show that $h^j(X_2,\mathcal{I}_Z \otimes\mathcal{L}[2]) = 0$, it is sufficient to show that $h^j(X_2,\mathcal{I}_{Z'\cup (2\overline{S}, X_2) \cup (2\widetilde{Z}, X_2)}\otimes \mathcal{L}[2]) = 0$.

We have
\[
\deg ((2\overline{S},H)\cup \widetilde{S}) = (r+2)x_1 + y_1 = 3\alpha = h^0(H, \mathcal{L}[2]_{|H}).
\]
By \Cref{a4.1} ({\cite[Theorem 2]{b}}) the pair $(H, \mathcal{L}[2]_{|H})$ is not secant defective, so 
\[
h^i(H, \mathcal{I}_{(\Tr_H(Z'\cup (2\overline{S}, X_2))) \cup \widetilde{S}} \otimes \mathcal{L}[2]_{|H}) = h^i(H, \mathcal{I}_{(2\overline{S},H)\cup \widetilde{S}} \otimes \mathcal{L}[2]_{|H}) = 0 \text{ for }i=0,1.
\]
Thus, by the Horace Differential Lemma for double points (\Cref{Differential Horace lemma}) applied with $W=Z'\cup (2\overline{S},X_2)$ and $\mathcal{R}=\mathcal{L}[2]$ in order to show that 
$h^j(X_2,\mathcal{I}_{Z'\cup (2\overline{S}, X_2)\cup (2\widetilde{Z},X_2)}\otimes \mathcal{L}[2]) = 0$, it is sufficient to prove that 
$h^j(X_2,\mathcal{I}_{\Res_H(Z'\cup (2\overline{S}, X_2)) \cup (2\widetilde{S},H)}\otimes \mathcal{L}[1]) = 0$.
Observe that $\Res_H(Z'\cup (2\overline{S}, X_2)) = Z' \cup \overline{S}$.
We have 
\begin{equation}\label{eq4}
h^1(X_2,\mathcal{I}_{Z'\cup (2\widetilde{S}, H)} \otimes \mathcal{L}[1]) = 0    
\end{equation}
as a consequence of \Cref{claim1} and \Cref{a3} applied with $u:=y_1$ and $z-x_1-y_1$ instead of $z$.
\begin{claim}\label{claim2}
    $z-x_1-y_1 \geq \lceil \alpha/(r+1) \rceil$.
\end{claim}
The computational proof of this claim is presented in \Cref{Section:claims}.

By \cref{eq4} we have $h^0(X_2,\mathcal{I}_{Z'\cup (2\widetilde{S}, H)} \otimes \mathcal{L}[1] ) = 3\alpha - (z-x_1-y_1)(r+3) - (r+2)y_1 = 3\alpha - (r+3)z + ((r+2)x_1+y_1) +  x_1 = x_1 + 6\alpha - (r+3)z$. If $j=1$ we get
$h^0(X_2,\mathcal{I}_{Z'\cup (2\widetilde{S}, H)} \otimes \mathcal{L}[1] ) \geq x_1$ and if $j=0$ we get $h^0(X_2,\mathcal{I}_{Z'\cup (2\widetilde{S}, H)} \otimes \mathcal{L}[1] ) < x_1$.
It follows from \Cref{lo3} that in order to get $h^j(X_2,\mathcal{I}_{Z'\cup (2\widetilde{S}, H)\cup \overline{S}} \otimes \mathcal{L}[1]) = 0$ it is sufficient to have
$h^0(X_2,\mathcal{I}_{Z'}\otimes \mathcal{L}[0]) = 0$. This is true by \Cref{claim2} and \Cref{lol1}.
\end{proof}

The bound of the previous lemma can be improved, as already remarked inside the proof, via a computational check. We preferred to write a separate statement to not make the previous proof too cumbersome.

\begin{lemma}\label{a4}
Let $X_2=Y\times \mathbb P^2$ where $(Y, \mathcal{L})$, $r$ and $\alpha$ are as in \Cref{NotationYandX}.  If the pair $(Y, \Ll)$ is not secant defective, then for all the following values of $r$ and $\alpha$ we have that the pair $(X_2,\mathcal{L}[2])$ is not secant defective.    \begin{itemize}
        \item $r=2,3$ and $\alpha \geq 60$, 
        \item $r=4$ and $\alpha \geq 98$, 
        \item $r=5$ and $\alpha \geq 133$, 
        \item $r=6$ and $\alpha \geq 176$, 
        \item $r=7$ and $\alpha \geq 231$, 
        \item $r\geq 8$ and $\alpha \geq \frac{1}{81} (27r^3+144r^2+210r+79)$. 
    \end{itemize}
\end{lemma}

\begin{proof} The present Lemma is an improvement of the bounds on the $\alpha$'s of \Cref{a4old}. The previous bounds were based on \Cref{claim1}  and \Cref{claim2} which one can computationally check that hold for all the finitely many values of $\alpha$ of the present Lemma not covered in \Cref{a4old}.
\end{proof}

\begin{lemma}\label{a5.0}
Let $(Y, \Ll)$, $X_2$ and $H$ be as in \Cref{NotationYandX} and $r$ and $\alpha$ be as in \Cref{a4}. Set $a:=  \lfloor 4\alpha/(r+2)\rfloor$, $b:= 4\alpha - (r+2)a$, $z:= \lceil\binom{5}{2}\alpha/(r+3)\rceil$ and $z':= z-a-b$. Let $Z'\subset X_2$ be a general union of $z'$ double points of $X_2$.
Let $B\subset H$ be a set of $b$ general points. If the pair $(Y, \Ll)$ is not secant defective, then $h^1(X_2,\Ii_{Z'\cup (2B,H)}\otimes \Ll[2]) =0$.
\end{lemma}

\begin{proof}
Let $z_1 = \lfloor 6\alpha/(r+3)\rfloor$ and $x_1, y_1$ be as in Claim~\ref{claim1} with $z_1$ instead of $z$.
Let $\overline{z} = z'-z_1 + x_1$, $\widetilde{z} = y_1$ and $z'' = z_1-x_1-y_1$.

\begin{claim}\label{claim3}
$\overline{z} \geq 0$.  
\end{claim}

The proof of the above claim is given in \Cref{Section:claims}.

Choose two disjoint subsets $\overline{Z}$ and $\widetilde{Z}$ of the support of $Z'$ consisting of $\overline{z}$ and $\widetilde{z}$ points, respectively. Let $Z''$ be the union of the remaining $z''$ double points of $Z'$ that are supported outside $\overline{Z}\cup \widetilde{Z}$. Let $\overline{S}$ and $\widetilde{S}$ be general disjoint subsets of $H$ with $\#\overline{S} = \overline{z}$ and $\#\widetilde{S} = \widetilde{z}$. By semicontinuity, in order to show that $h^1(X_2, \Ii_{Z'\cup (2B,H)}\otimes\Ll[2]) = 0$ it is sufficient to establish the equality $h^1(X_2,\Ii_{Z''\cup (2\overline{S},X_2)\cup (2\widetilde{Z}, X_2)\cup (2B,H)}\otimes \Ll[2]) = 0$.

{\hangindent=1cm
\begin{claim}\label{claim4}
$h^1(H,\mathcal{I}_{(2\overline{S}, H)\cup \widetilde{S} \cup (2B, H)}, \mathcal{L}[2]_{|H}) = 0$.
\end{claim}
\begin{proof}[Proof of \Cref{claim4}]
By Lemma~\ref{a4.1} ({\cite[Theorem 2]{b}}) it is enough to show that $\deg ((2\overline{S}, H)\cup \widetilde{S} \cup (2B, H))\leq h^0(H,\mathcal{L}[2]_{|H}) = 3\alpha$.
We have 
\begin{align*}
&\deg ((2\overline{S}, H)\cup \widetilde{S} \cup (2B, H))  = (r+2)\overline{z} + \widetilde{z} + (r+2)b  \\
&= (r+2)(z'-z_1 + x_1) + y_1 + (r+2)b =(r+2)(z'-z_1+b) + (r+2)x_1 + y_1 \\
&=(r+2)(z'-z_1 + b) + 3\alpha 
\end{align*}
so it is sufficient to show that
\begin{equation}\label{eq7}
(r+3)(r+2)(z'-z_1 + b)\leq 0.
\end{equation}
We have
\[
(r+3)(r+2)(z'-z_1 + b) = (r+3)(r+2)(z-a) - (r+3)(r+2)z_1
\]
\[
 \leq (r+2)(10\alpha + (r+2)) - (r+3)(4\alpha - (r+1)) - (r+2)(6\alpha - (r+2)) = 3r^2+12r + 11 - 4\alpha
\]
which is negative by the lower bounds on $\alpha$ and $r$.
\end{proof}
}

{\hangindent=1cm
\begin{claim}\label{claim5}
$h^1(X_2,\mathcal{I}_{Z'' \cup \overline{S} \cup (2\widetilde{S}, H)}, \mathcal{L}[1]) = 0$.
\end{claim}
\begin{proof}[Proof of \Cref{claim5}]
The integers $z_1, x_1$ and $y_1$ satisfy the inequalities of \Cref{claim1} (with $z_1$ instead of $z$). Therefore, we have $h^1(X_2,\mathcal{I}_{Z''  \cup (2\widetilde{S}, H)}, \mathcal{L}[1]) = 0$ by \Cref{a3} \cref{seconda} applied with $z=z'' = z_1-x_1-y_1$ and $u = \widetilde{z} = y_1$. It follows that $h^0(X_2,\mathcal{I}_{Z''  \cup (2\widetilde{S}, H)}, \mathcal{L}[1]) = 3\alpha - (z_1-x_1-y_1)(r+3) - (r+2)y_1 = (x_1+6\alpha - (r+3)z_1)$
which is at least $x_1$. On the other hand we have $\overline{z} \leq \overline{z} + b = z'-z_1 + b + x_1 \leq x_1$ by \cref{eq7}.
It follows from \Cref{lo3} that to conclude the proof of the claim it is sufficient to show that $h^0(X_2, \mathcal{I}_{Z''}\otimes \mathcal{L}[0]) = 0$. This follows from \Cref{lol1} and \Cref{claim2} (with $z_1$ instead of $z$).
\end{proof}
}
By \Cref{claim4} and \Cref{claim5}, the lemma follows from Horace Differential Lemma for double points (\Cref{Differential Horace lemma}) applied with $W= Z''\cup (2\overline{S},X_2)\cup (2B,H)$ and $\Rr = \Ll[2]$.
\end{proof}

\begin{lemma}\label{a1.2}
Let $(Y, \Ll), X_2, r$ and $\alpha$ be as in {\Cref{a4}}. 
 If $Z$ is a general union of $z$ double points of $X_2$ {and the pair $(Y, \Ll)$ is not secant defective, then}  $h^0(X_2,\Ii_Z\otimes \Ll[1]) \le \max \{0,3\alpha -(r+2)z\}$. 
\end{lemma}

\begin{proof}
Let $H = Y\times \PP^1$ be a divisor of $X_2$ as in \Cref{NotationYandX}.
Write $Z$ as the following disjoint union $Z=(2\overline Z,X_2)\cup (2\widetilde{Z},X_2)\cup Z'$ {where $\overline Z$ and $\widetilde Z$ are disjoint subsets of the support of $Z$ whose cardinalities are specified below}. 
Take a general $\overline S\cup \widetilde{S}\subset H$ such that $\#\overline S= \#\overline Z$, $\#\widetilde{S}=\#\widetilde{Z}$ and $\overline S\cap \widetilde{S}=\emptyset$. By semicontinuity it is enough to show that $h^0(X_2,\mathcal{I}_{Z' \cup (2 \overline S, X_2) \cup (2 \widetilde{Z}, X_2)} \otimes \mathcal{L}[1])\le \max \{0,3\alpha -(r+2)z\}$. Let $$a:= 2\lfloor \alpha/(r+2)\rfloor \hbox{ and } b:=2\alpha-(r+2)a.$$

\begin{itemize}
    \item 
Assume for the moment $z\ge a+b$. Let $\# \overline Z=a$ and $\# \widetilde{Z}= b$. 
Note that $0\le b\le 2r+2$. 
\Cref{a1} applied to $(2\overline S,H)$ gives $h^1(H,\Ii_{(2\overline S,H)}\otimes \Ll[1]_{|H})=0$. Now if we add the general $b$ simple points of $\widetilde{S}$ to $(2\overline S,H)$ then we get also
 $h^i(H,\Ii_{(2\overline S,H)\cup \widetilde{S}}\otimes \Ll[1]_{|H})=0$ for $i=0,1$ since $\deg((2\overline S,H)\cup \widetilde{S})=(r+2)a + b= 2 \alpha = h^0(H,\Ll[1]_{|H})$.
By the Horace Differential Lemma for double points (\Cref{Differential Horace lemma}) applied with $W=Z' \cup (2 \overline S, X_2)$ and $\mathcal{R}= \mathcal{L}[1]$, we have $h^0(X_2,\Ii _{Z' \cup (2 \overline S, X_2) \cup (2 \widetilde{Z}, X_2)}\otimes \Ll[1]) \leq h^0(X_2,\Ii_{Z' \cup \overline S\cup (2\widetilde{S},H)}\otimes \Ll[0])$. Since the pair $(Y,\Ll)$ is not secant defective, $Z'$ is general in $X_2$ and $\overline S\cup \widetilde{S}$ is general in $H$, it follows from \Cref{lol1} that $h^0(X_2,\Ii_{Z' \cup \overline S\cup (2\widetilde{S},H)}\otimes \Ll[0])
 =\max \{0,\alpha -a -(r+1)(z-a-b) -(r+1)b\} = \max \{0, \alpha + ar - (r+1)z\} \leq \max\{0, \alpha + \alpha\frac{2r}{r+2} - (r+1)z\}$.  
Now since it is easy to check that 
 \[
 \max\Bigl\{0, \alpha + \alpha\frac{2r}{r+2} - (r+1)z\Bigr\} \leq \max\{0, 3\alpha - (r+2)z\}
 \]
we are done in the case $z\geq a+b$.

\item

Now assume $a\le z<a+b$. We take $Z'=\emptyset$,  $\# \overline Z=a= 2\lfloor \alpha/(r+2)\rfloor$ and $\# \widetilde{Z}= z-a$. 
 The same argument of the previous case 
leads to
\begin{equation}\label{eqa1.2.1}
    h^1(H, \Ii_{(2\overline S, H)\cup \widetilde{S}}\otimes \Ll[1]_{|H}) = 0.
\end{equation} 
Let $\beta = h^0(X_2,\Ii_{(2\widetilde{S}, H)}\otimes \Ll[0]).$ By \Cref{lol1} we have $\beta = \max\{0, \alpha - (r+1)(z-a)\} = \alpha - (r+1)(z-a)$. 
Therefore, $h^1(X_2,\Ii_{(2\widetilde{S}, H)}\otimes \Ll[0]) = 0$.
We have $h^0(X_2,\Ii_{\Res_{H}(2\widetilde{S}, H)} \otimes \Ll[-1]) = 0 \leq \beta - a$ where we skip the routine verification of the inequality.
It follows from \Cref{lo3} that 
\begin{equation}\label{eqa1.2.2}
h^1(X_2,\Ii_{(2\widetilde{S}, H)\cup \overline S}\otimes \Ll[0])) = 0.
\end{equation}
From the Horace Differential Lemma for double points (\Cref{Differential Horace lemma}) applied to $W=(2\overline S, X_2)$ and $\mathcal{R}= \mathcal{L}[1]$ combined with semicontinuity, \cref{eqa1.2.1} and \cref{eqa1.2.2} we get $h^1(X_2,\Ii_{Z}\otimes \Ll[1]) = 0$ and therefore, $h^0(X_2,\Ii_Z\otimes \Ll[1])  = 3\alpha - (r+3)z \leq  3\alpha - (r+2)z$.

\item

Now assume $z<a$. We take $Z'=\widetilde{Z}=\emptyset$ and $\#\overline Z=z$. 
We have $h^1(H, \Ii_{(2\overline{S}, H)}\otimes \Ll[1]_{|H}) = 0$ and $h^1(X_2,\Ii_{\overline{S}} \otimes \Ll[0]) =0$. Therefore, by the standard Horace Lemma we have $h^1(X_2,\Ii_{(2\overline{S}, X_2)} \otimes \Ll[1])=0$. Thus, $h^0(X_2,\Ii_{Z} \otimes \Ll[1]) = 3\alpha - (r+3)z \leq 3\alpha - (r+2)z$.
\end{itemize}
\vspace{-0,5cm}
\end{proof}

At this stage, we have gathered all essential elements to formulate the main theorem within this paper.

\begin{theorem}
    \label{a5}
Let $(Y, \Ll), X_2, r$ and $\alpha$ be as in \Cref{a4}. If $t\ge 2$  and the pair $(Y, \Ll)$ is not secant defective,  then $(X_2,\Ll[t])$ is not secant defective.
\end{theorem}

\begin{proof}
Since the case $t=2$ was already proved in \Cref{a4}, we may assume $t>2$ and proceed by induction. Namely assume that   $(X_2,\Ll[x])$ is non secant defective for  $2\le x<t$. It is sufficient to prove
that either $h^1(X_2,\Ii_Z\otimes \Ll[t])=0$ or $h^0(X_2,\Ii_Z\otimes \Ll[t])=0$ for a general union $Z\subset X_2$ of $z$ double points of $X_2$ for all $z\in \{\lfloor \binom{t+2}{2}\alpha/(r+3)\rfloor, \lceil\binom{t+2}{2}\alpha/(r+3)\rceil\}$.
Set $j:= 1$ if $z=\lfloor \binom{t+2}{2}\alpha/(r+3)\rfloor$ and $j:=0$ otherwise. We need to prove that $h^j(X_2,\Ii_Z\otimes \Ll[t])=0$.
Let $H=X\times \PP^1$ be a divisor of $X_2$ as in \Cref{NotationYandX}. By 
\Cref{a4.1} ({\cite[Theorem 2]{b}}) the pairs $(H,\Ll[t]_{|H})$ and $(H,\Ll[t-1]_{|H})$ are not secant defective.
Set $$a:= \lfloor (t+1)\alpha/(r+2)\rfloor \hbox{ and  } b:= (t+1)\alpha - (r+2)a.$$

\begin{claim}\label{claim11}
$z\ge a+b$
\end{claim}

See \Cref{Section:claims} for the proof of \Cref{claim11}.

Decompose $Z$ into the disjoint union $Z=Z'\cup (2\overline{Z}, X_2)\cup (2\widetilde{Z}, X_2)$ where $\#\overline{Z} = a$ and $\#\widetilde{Z} = b$. Take a general $\overline{S}\cup \widetilde{S}\subset H$ such that $\#\overline{S} = a$, $\#\widetilde{S} = b$ and $\overline{S}\cap \widetilde{S} = \emptyset$. 
By semicontinuity and Horace Differential Lemma for double points (\Cref{Differential Horace lemma}) applied to $W=Z'\cup (2\overline{S}, X_2)$ and $\Rr = \Ll[t]$ we have that $h^j(X_2,\mathcal{I}_Z \otimes \mathcal{L}[t]) \leq h^j(X_2,\Ii_{Z'\cup \overline{S}\cup (2\widetilde{S}, H)}\otimes \Ll[t-1]) + h^j(H,\Ii_{(2\overline{S}, H) \cup \widetilde{S}} \otimes \Ll[t]_{|H})$.
In order to show that $h^j(X_2,\Ii_Z\otimes \Ll[t])=0$, it is now sufficient  to prove that $h^j(X_2,\Ii_{Z'\cup \overline{S}\cup (2\widetilde{S}, H)}\otimes \Ll[t-1]) = h^j(H,\Ii_{(2\overline{S}, H) \cup \widetilde{S}} \otimes \Ll[t]_{|H})=0$.
By \Cref{a4.1} the pair $(H,\Ll[t]_{|H})$ is not secant defective, moreover here the degree of the scheme is equal to $h^0(H, \Ll[t]_{|H})$ so $h^i(H,\Ii_{(2\overline{S},H)\cup \widetilde{S}}\otimes \Ll[t]_{|H})=0$, $i=0,1$. 
Thus it is sufficient to prove $h^j(X_2,\Ii_{Z'\cup \overline{S}\cup (2\widetilde{S},H)}{\otimes \Ll[t-1]})=0$.

{\hangindent=1cm
\begin{claim}\label{claim10}
$h^1(X_2,\Ii_{Z'\cup (2\widetilde{S},H)}\otimes \Ll[t-1])=0$.
\end{claim}
\begin{proof}[Proof of \Cref{claim10}]
If $b=0$ there are no double points on $H$ (i.e. $\widetilde S$ is empty), so we just use that $(X_2,\Ll[t-1])$ is not secant defective and $(r+3)(z-a) \le \binom{t+1}{2}\alpha$, because $(r+3)z \le \binom{t+2}{2}\alpha +r+2$, $(r+2)a =(t+1)\alpha$ and $a\ge r+2$. 
Now assume $b>0$. If $t=3$, then \Cref{claim10} is true by  \Cref{a5.0} applied with $B=\widetilde{S}$. Thus we may assume $t\ge 4$. Set $$c:= \lfloor t\alpha/(r+2)\rfloor -b.$$ Decompose $Z'$ into the disjoint union $Z'=Z''\cup (2\widehat{Z}, X_2)$ where $\#\widehat{Z}$ is specified below depending on the value of $c$. Let $\widehat{S}$ be a general subset of $H$ such that $\widehat{S}\cap \widetilde{S} = \emptyset$ and $\#\widehat{Z} = \#\widehat{S}$.

\begin{enumerate}[label=(\arabic*)]
\item\label{quada1} Assume that $c\le z-a-b$.
Let $\#\widehat{Z} = c$. Since $(H,\Ll[t-1]_{|H})$ is not secant defective, $h^1(H,\Ii_{(2\widehat{S},H)\cup (2\widetilde{S},H)}\otimes \Ll[t-1]_{|H})=0$. Thus by semicontinuity and the standard Horace Lemma, in order to prove \Cref{claim10} it is sufficient to prove $h^1(X_2,\Ii_{Z''\cup \widehat{S}}\otimes \Ll[t-2]) =0$.
One can compute that $\deg (Z''\cup \widehat{S}) \le \binom{t}{2}\alpha$.  Since $t\ge 4$ the inductive assumption gives $h^1(X_2,\Ii_{Z''}\otimes \Ll[t-2]) =0$. Therefore, $h^0(X_2,\Ii_{Z''}\otimes \Ll[t-2]) =\binom{t}{2}\alpha -(r+3)(z-a-b-c)$. Hence to prove that $h^1(X_2,\Ii_{Z''\cup \widehat{S}}\otimes \Ll[t-2]) =0$ it is sufficient by \Cref{lo3} to prove that $h^0(X_2,\Ii_{Z''}\otimes \Ll[t-3]) \le \binom{t}{2}\alpha -(r+3)(z-a-b-c)-c$.

\begin{itemize}
    \item If $t\ge 5$, the last inequality is true, because $(X_2,\Ll[t-3])$ is not secant defective {and $c\leq c+b\leq t\alpha / (r+2) \leq (t-1)\alpha$.} 
    \item  Now assume $t=4$ with the assumption $z\ge a+b+c$. We have $0\le b\le r+1$, $(r+2)a +b =5\alpha$, $|(r+3)z -15\alpha|\le r+2$ and $4\alpha -r-1 \le (r+2)(b+c)\le 4\alpha$ and it is sufficient to prove
    that $h^0(X_2,\Ii_{Z''}\otimes \Ll[1]) \le 6\alpha -(r+3)(z-a-b-c)-c$. 
    Lemma \ref{a1.2} gives $h^0(X_2,\Ii_{Z''}\otimes \Ll[1]) \le \max\{0,3\alpha -(r+2)(z-a-b-c)\}$. Thus it is sufficient to prove that $(z-a-b-c) + c\le 3\alpha$.  
\end{itemize}

\item\label{quada2} If $a+b\le z<a+b+c$ to get \Cref{claim10} using the proof in \cref{quada1} we take $\#\widehat{Z} =z-a-b$ and $Z''=\emptyset$. 
\end{enumerate}
\vspace{-0,5cm}
\end{proof}
}

\Cref{claim10} gives $h^0(X_2,\Ii_{Z'\cup (2\widetilde{S},H)}\otimes \Ll[t-1])=\binom{t+1}{2}\alpha -(r+3)(z-a-b)-(r+2)b$. By \Cref{lo3} and the generality of $\overline{S}$ in $H$, in order to conclude the proof of \Cref{a5} it is sufficient to prove that $h^0(X_2,\Ii_{Z'}\otimes \Ll[t-2]) \le \max \{0,\binom{t+1}{2}\alpha -(r+3)(z-a-b)-(r+2)b-a\}$. If $t\ge 4$ we can use the inductive assumption. Namely, since $(X_2,\Ll[t-2])$ is not secant defective, we have $h^0(X_2,\Ii_{Z'}\otimes \Ll[t-2]) = \max \{0, \binom{t}{2}\alpha - (r+3)(z-a-b)\}$ so it is sufficient to observe that $\binom{t}{2}\alpha \leq \binom{t+1}{2}\alpha - (r+2)b-a$.
Now assume $t=3$. We need to check the inequality $h^0(X_2,\Ii_{Z'}\otimes \Ll[1]) \le \max \{0,6\alpha -(r+3)(z-a-b) -(r+2)b-a\}$. 
 \Cref{a1.2} gives $h^0(X_2,\Ii_{Z'}\otimes \Ll[1]) \le \max \{0,3\alpha -(r+2)(z-a-b)\}$. Thus it is sufficient to prove the inequality $3\alpha \ge z-a-b +(r+2)b +a =z +(r+1)b$. We have $(r+3)(z+(r+1)b) \leq 10\alpha + (r+2) + (r+1)^2(r+3)$
so it is enough to show that $3\alpha(r+3) \geq 10\alpha + (r+2)+(r+1)^2(r+3)$, or equivalently, that
\[
\alpha(3r-1) \geq r^3+5r^2+8r+5.
\]
This inequality holds by the assumed lower bounds on $\alpha$ from \Cref{a4}.
\end{proof}

\section{Applications}\label{Applications}

\begin{proposition}\label{P1orP2}
Assume that $n\in \{1,2\}$ and $t\geq 2.$
If $(Y, \mathcal{L})$ satisfies the assumptions of \Cref{a5}, then so does $(Y\times \PP^n, \mathcal{L}[t])$.
\end{proposition}
\begin{proof}
    Observe that if $r$ and $\alpha$ are as in \Cref{a4}, then $\alpha > (r+1)^2$. Therefore, $(Y\times \PP^n, \mathcal{L}[t])$ is not secant defective by \Cref{a4.1} ({\cite[Theorem 2]{b}}) if $n=1$ and by \Cref{a5} if $n=2$. One may verify that $r+n$ and $\binom{t+n}{n} \alpha$ satisfy the inequalities in \Cref{a4}.
\end{proof}

\begin{proposition}
    \label{twofactors_nd}
For every $t\geq 3$ the pair $(\PP^2\times \PP^2, \mathcal{O}_{\PP^2\times \PP^2}(2, t))$ is not secant defective.
\end{proposition}
\begin{proof}
    For $t\geq 10$ we have $h^0(\PP^2, \mathcal{O}_{\PP^2}(t)) \geq 60$. Furthermore, for such values of $t$, by Alexander-Hirschowitz theorem (cf. \cite{ah}) 
    the pair $(\PP^2, \mathcal{O}_{\PP^2}(t))$ is not secant defective. Therefore, by \Cref{a5} it is sufficient to prove the statement for $3\leq t \leq 9$. Due to Terracini's lemma these cases can be checked on a computer to be not secant defective. We did it with Macaulay2 (\cite{M2}).
\end{proof}

\begin{proposition}
\label{threefactors_nd}
For every $t\geq 2$ the pair $(\PP^2\times\PP^2\times \PP^2, \mathcal{O}_{\PP^2\times \PP^2\times \PP^2}(2,2,t))$ is not secant defective.
\end{proposition}
\begin{proof}
For $t\geq 5$ we have $h^0(\PP^2\times \PP^2, \mathcal{O}_{\PP^2\times \PP^2}(2,t)) \geq 98$. Furthermore, for such values of $t$, the pair $(\PP^2\times \PP^2, \mathcal{O}_{\PP^2\times \PP^2}(2,t))$ is not secant defective by \Cref{twofactors_nd}.
Therefore, by \Cref{a5} it is sufficient to prove the Proposition for $2\leq t\leq 4$. Using Terracini's lemma one may check computationally that the following pairs are not secant defective
\begin{enumerate}[label=\arabic*.]
    \item \label{222:221}$(\PP^2\times \PP^2\times \PP^2, \mathcal{O}_{\PP^2\times \PP^2\times \PP^2}(2,2,1))$;
    \item  \label{222:222}$(\PP^2\times \PP^2\times \PP^2, \mathcal{O}_{\PP^2\times \PP^2\times \PP^2}(2,2,2))$;
    \item \label{221:223}$(\PP^2\times \PP^2\times \PP^1, \mathcal{O}_{\PP^2\times \PP^2\times \PP^1}(2,2,3))$;
    \item \label{221:224}$(\PP^2\times \PP^2\times \PP^1, \mathcal{O}_{\PP^2\times \PP^2\times \PP^1}(2,2,4))$.
\end{enumerate}
Again we did it with Macaulay2 (\cite{M2}).

Let $X=\PP^2\times \PP^2\times \PP^2$ and $H = \PP^2\times \PP^2\times \PP^1\subset X$ be a divisor. Let $\mathcal{L} = \mathcal{O}_{\PP^2\times \PP^2}(2,2)$. Using this notation, we need to prove that the pairs $(X, \mathcal{L}[t])$ with $t=3,4$ are not secant defective. 
Let $z\in \{\lfloor h^0(X, \mathcal{L}[t])/(\dim X+1)\rfloor = \lfloor \frac{36}{7}\binom{t+2}{2}\rfloor, \lceil \frac{36}{7}\binom{t+2}{2}\rceil\}$ and $Z$ be the union of $z$ general double points of $X$. Let $j=1$ if $z=\lfloor \frac{36}{7}\binom{t+2}{2}\rfloor$ and $j=0$ otherwise. It is sufficient to prove that $h^j(X, \mathcal{I}_Z\otimes \mathcal{L}[t]) = 0$.
Let $\overline{z} = 6(t+1)$ and $\overline{S}\subset H$ be a general subset with $\#\overline{S} = \overline{z}$.
Decompose $Z$ into the disjoint union $Z = Z'\cup (2\overline{Z}, X)$ where $\overline{Z}$ is a subset of the support of $Z$ of cardinality $\overline{z}$. By the standard Horace Lemma and semicontinuity, it is now sufficient to show that 
$h^j(X, \mathcal{I}_{Z'\cup \overline{S}}\otimes \mathcal{L}[t-1]) = h^j(H, \mathcal{I}_{(2\overline{S}, H)}\otimes \mathcal{L}[t]_{|H}) = 0$.
We have $\deg (2\overline{S}, H) = 6\overline{z} = h^0(H, \mathcal{L}[t]_{|H})$. Therefore, $h^j(H, \mathcal{I}_{(2\overline{S}, H)}\otimes \mathcal{L}[t]_{|H}) = 0$ since $(H, \mathcal{L}[t]_{|H})$ is not secant defective as mentioned above in \cref{221:223} and \cref{221:224}

On the other hand, $\deg Z' = 7(z-\overline{z}) \leq   36 \binom{t+1}{2} = h^0(X, \mathcal{L}[t-1])$.
Therefore, since the pair $(X, \mathcal{L}[t-1])$ is not secant defective (for $t=3$ it was mentioned above in \cref{222:222} and for $t=4$ it follows after we prove this Proposition for $t=3$), we obtain  that $\beta:= h^0(X, \mathcal{I}_{Z'}\otimes \mathcal{L}[t-1]) = 36 \binom{t+1}{2} - 7(z-\overline{z})$. In particular, $h^1(X, \mathcal{I}_{Z'}\otimes \mathcal{L}[t-1]) = 0$. Furthermore, $\mathrm{Res}_H(Z') = Z'$ since $ Z'$ is general in $X$. Moreover $h^0(X, \mathcal{I}_{Z'}\otimes\mathcal{L}[t-2]) = \max\{0, 36\binom{t}{2} - 7(z-\overline{z})\}$ since the pair $(X, \mathcal{L}[t-2])$ is not secant defective by \cref{222:221} and \cref{222:222}
It follows from \Cref{lo3} that to conclude the proof it is sufficient to check that $36\binom{t}{2} -7(z-\overline{z}) \leq \beta-\overline{z} = 36 \binom{t+1}{2} - 7z + 6\overline{z}$. This is equivalent to the inequality $\overline{z} \leq 36 t$ which holds since $\overline{z} = 6(t+1) \leq 30 \leq 36t$.
\end{proof}

\begin{theorem}\label{prod:P2}
The pair $((\PP^2)^k, \mathcal{O}_{(\PP^2)^k}(t_1,\ldots, t_k))$ with $t_1,\ldots, t_k\geq 2$ is secant defective if and only if one of the following holds:
\begin{itemize}
    \item $k=1$ and $t_1\in \{2,4\}$
    \item $k=2$ and $t_1=t_2 = 2$.
\end{itemize}
\end{theorem}
\begin{proof}
    If $k=1$, the secant defective cases are described by  Alexander-Hirschowitz theorem (cf. \cite{ah}).
    The fact that $(\PP^2\times \PP^2, \mathcal{O}_{\PP^2\times \PP^2}(2,2))$ is secant defective was already known (cf. \cite{cgg}).
    The pairs $(\PP^2\times \PP^2, \mathcal{O}_{\PP^2\times \PP^2}(2,t))$ with $t\geq 3$ are not secant defective by \Cref{twofactors_nd}. The pairs $(\PP^2\times \PP^2, \mathcal{O}_{\PP^2\times \PP^2}(t_1,t_2))$ with $t_1,t_2\geq 3$ are not secant defective by \cite{go}.

    The pairs $(\PP^2\times \PP^2\times \PP^2, \mathcal{O}_{\PP^2\times \PP^2\times \PP^2}(2,2,t))$ with $t\geq 2$ are not secant defective by \Cref{threefactors_nd}. The pairs $(\PP^2\times \PP^2\times \PP^2, \mathcal{O}_{\PP^2\times \PP^2\times \PP^2}(t_1,t_2,t_3))$ with $t_1,t_2\geq 3$ and $t_3\geq 2$ are not secant defective by \cite{go} and \Cref{a5} since
    $h^0(\PP^2\times \PP^2, \mathcal{O}_{\PP^2\times \PP^2}(t_1,t_2)) \geq h^0(\PP^2\times \PP^2, \mathcal{O}_{\PP^2\times \PP^2}(3,3)) = 100 \geq 98$.

   Therefore we have that the pair $(\PP^2\times \PP^2\times \PP^2, \mathcal{O}_{\PP^2\times \PP^2\times \PP^2}(t_1,t_2,t_3))$
    is not secant defective for all $t_1,t_2,t_3 \geq 2$. Since $h^0(\PP^2\times \PP^2\times \PP^2, \mathcal{O}_{\PP^2\times \PP^2\times \PP^2}(t_1,t_2,t_3)) \geq h^0(\PP^2\times \PP^2\times \PP^2, \mathcal{O}_{\PP^2\times \PP^2\times \PP^2}(2,2,2)) = 216 \geq 176$ we conclude using \Cref{P1orP2} and induction on $k$ that $((\PP^2)^k, \mathcal{O}_{(\PP^2)^k}(t_1,\ldots, t_k))$ with $t_1,\ldots, t_k\geq 2$ is not secant defective for every $k\geq 3$.
\end{proof}

\begin{proposition}
    \label{112}
    If $s_1,s_2,t\geq 2$ the pair $(\PP^1\times\PP^1\times\PP^2, \mathcal{O}_{\PP^1\times \PP^1\times \PP^2}(s_1,s_2,t))$ is  secant defective if and only if $s_1=s_2=t=2$.
\end{proposition}
\begin{proof}
    See \cite{cgg} for the secant defective case. If either $t\geq 3$ or $t=2$ and $s_1$ and $s_2$ are not both even the pair is not secant defective by \cite{BaurDraisma} together with  \Cref{a4.1} ({\cite[Theorem 2]{b}}).
    In what follows we assume that $t=2$ and $s_1=2a$, $s_2=2b$ for some positive integers $a,b$. Suppose that both $a$ and $b$ are larger than $1$. Then the pair $(\PP^1\times \PP^1, \mathcal{O}_{\PP^1\times \PP^1}(2a,2b))$ is not secant defective by \cite{lp}.
    Therefore, by \Cref{a5}, if $(2a+1)(2b+1) \geq 60$ then the pair $(\PP^1\times\PP^1\times\PP^2, \mathcal{O}_{\PP^1\times \PP^1\times \PP^2}(2a,2b,2))$ is not secant defective. If we assume that $a\geq b$, the remaining cases (still with $a,b>1$) are 
    $(a,b) \in \{(2,2), (3,2), (4,2), (5,2),(3,3)\}$. One may verify with a computer aid that in all these cases the pair
$(\PP^1\times\PP^1\times\PP^2, \mathcal{O}_{\PP^1\times \PP^1\times \PP^2}(2a,2b,2))$ is not secant defective (we did it via Macaulay2 \cite{M2}).
    We verified computationally (with \cite{M2}) that also the pairs $(\PP^1\times\PP^1\times\PP^2, \mathcal{O}_{\PP^1\times \PP^1\times \PP^2}(2a,2,2))$ with $2\leq a \leq 5$ are not secant defective. 
    
    Therefore, in what follows we assume that $a\geq 6$ and $b=1$. 
    
    Let $z\in \{\lfloor 18(2a+1)/5\rfloor , \lceil 18(2a+1)/5\rceil\}$ and let $j=1$ if $z=\lfloor 18(2a+1)/5\rfloor$ and $j=0$ otherwise. Let $Z$ be a general union of $z$ double points of $\PP^1\times \PP^1\times\PP^2$.
    In order to show that the pair $(\PP^1\times\PP^1\times\PP^2, \mathcal{O}_{\PP^1\times \PP^1\times \PP^2}(2a,2,2))$ is not secant defective it is sufficient to show that $h^j(\PP^1\times\PP^1\times\PP^2, \mathcal{I}_Z(2a,2,2)) = 0$.

    Let $\overline{z} = \lfloor 9(2a+1)/4\rfloor$, $\widetilde z = 9(2a+1)-4\overline{z}$ and $z'=z-\overline{z}-\widetilde{z}$. Decompose $Z$ into the disjoint union 
    \[
    Z = Z' \cup (2\overline{Z}, \PP^1\times\PP^1\times \PP^2)\cup (2\widetilde{Z}, \PP^1\times\PP^1\times \PP^2)
    \]
    where $\#\overline{Z} = \overline z$ and $\#\widetilde{Z} = \widetilde{z}$. Let $H=\PP^1\times \PP^1\times \PP^1$
    and $\overline S, \widetilde S$ be general disjoint subsets of $H$ with $\# \overline S = \overline z$ and $\#\widetilde S = \widetilde z$.
    By semicontinuity and Horace Differential Lemma for double points (\Cref{Differential Horace lemma}) applied with $W=Z'\cup (2\overline S, \PP^1\times \PP^1\times \PP^2)$ in order to conclude that $h^j(\PP^1\times\PP^1\times\PP^2, \mathcal{I}_Z(2a,2,2)) = 0$
    it is sufficient to show that
    \[
    h^j(\PP^1\times\PP^1\times\PP^2, \mathcal{I}_{Z'\cup \overline{S} \cup (2\widetilde S, H)}(2a,2,1)) = h^j(H, \mathcal{I}_{(2\overline S, H)\cup \widetilde S}(2a,2,2)) = 0.
    \]  
    The pair $(\PP^1\times\PP^1\times \PP^1, \mathcal{O}_{\PP^1\times\PP^1\times\PP^1}(2a,2,2)$ is not secant defective by \cite{lp}.
    Since 
    \[
    \deg (2\overline S, H)\cup \widetilde S = 4 \overline z+ \widetilde z
    = h^0(H, \mathcal{O}_H(2a,2,2))
    \]
    and $\overline S$, $\widetilde S$ are general we conclude that $h^j(H, \mathcal{I}_{(2\overline S, H)\cup \widetilde S}(2a,2,2)) = 0$.

    \begin{claim}\label{claim7}
    We have 
    \begin{itemize}
        \item $h^1(\PP^1\times\PP^1\times\PP^2, \mathcal{I}_{Z'\cup (2\widetilde S, H)}(2a,2,1)) = 0$.        
        \item $h^0(\PP^1\times\PP^1\times \PP^2, \mathcal{I}_{Z'}(2a,2,0)) = 0$.
    \end{itemize}
    \end{claim}
    For the proof of \Cref{claim7} see \Cref{Section:claims}.
    As a result of \Cref{claim7} and \Cref{lo3} we obtain $h^j(\PP^1\times\PP^1\times\PP^2, \mathcal{I}_{Z'\cup \overline{S} \cup (2\widetilde S, H)}(2a,2,1)) = 0$.
    \end{proof}

\begin{proposition}
    \label{122}
For every $t\geq 2$ the pair $(\PP^1\times \PP^2\times \PP^2, \mathcal{O}_{\PP^1\times \PP^2\times \PP^2}(t, 2,2 ))$ is not secant defective.
\end{proposition}
\begin{proof}
    The pairs $(\PP^1\times \PP^2\times \PP^2, \mathcal{O}_{\PP^1\times \PP^2\times \PP^2}(t, 2,2 ))$ with $t=2,3$ are not secant defective by a computer check (cf. \cite{M2}). We proceed by induction. In what follows we assume that $t\geq 4$ and the pair $(\PP^1\times \PP^2\times \PP^2, \mathcal{O}_{\PP^1\times \PP^2\times \PP^2}(t-2, 2,2 ))$ is not secant defective.

    To simplify notation, let $Y=\PP^2\times \PP^2$, $\mathcal{L} = \mathcal{O}_Y(2,2)$ and $X=\PP^1\times Y$.
    Let $z=h^0(X, \Ll[t]) / (\dim X + 1)= 6(t+1)$ and $Z$ be a general union of $z$ double points of $X$. We need to show that $h^1(X, \mathcal{I}_Z\otimes \Ll[t]) = 0$. Let $\overline{Z}$ and $\widetilde{Z}$ be disjoint subsets of the support of $Z$ with $\# \overline{Z} = \#\widetilde{Z} = 6$. Decompose $Z$ into the disjoint union $Z = Z'\cup (2\overline{Z}, X)\cup (2\widetilde{Z}, X)$. Let $o$ be a point of $\PP^1$ and $H$ be the divisor $\{o\}\times Y  \subseteq \PP^1\times \PP^2\times \PP^2$. Let $\overline{S}$ and $\widetilde{S}$ be general disjoint subsets of $H$ with $\#\overline{S} = \#\widetilde{S}=6$. By semicontinuity and Horace Differential Lemma for double points (\Cref{Differential Horace lemma}) applied with $W=Z'\cup (2\overline{S}, X)$ it is now sufficient to show that
    \[
    h^1(X, \Ii_{Z'\cup \overline{S}\cup (2\widetilde{S}, H)}\otimes \Ll[t-1]) = h^1(H, \Ii_{(2\overline{S}, H)\cup \widetilde{S}}\otimes \mathcal{L}[t]_{|H}
    ) = 0.
    \]
    We have $\deg  (2\overline{S}, H)\cup \widetilde{S} = 36 = h^0(Y, \mathcal{O}_Y(2,2))$. Since the $6$-th secant variety of $(Y, \mathcal{O}_Y(2,2))$ is not defective by \cite{CGG08} we conclude that $h^1(H, \Ii_{(2\overline{S}, H)\cup \widetilde{S}}\otimes \mathcal{L}[t]_{|H}) = 0$.

    In order to show that $h^1(X, \Ii_{Z'\cup \overline{S}\cup (2\widetilde{S}, H)}\otimes \Ll[t-1]) = 0$, it is sufficient by the standard Horace Lemma to show that $h^1(X, \Ii_{Z'}\otimes \Ll[t-2]) = h^1(H, \Ii_{\overline{S}\cup (2\widetilde{S}, H)}\otimes \mathcal{L}[t-1]_{|H})=0$. The first of these cohomology groups vanishes since by induction the pair $(X, \Ll[t-2])$ is not secant defective. Since $\# \widetilde{S} = \# \overline{S}$, the second one vanishes as already shown above.
\end{proof}

\begin{theorem}\label{thm:p1andp2}
If $s_1, \ldots, s_j, t_1, \ldots, t_k\geq 2$, then the pair 
\[
((\PP^1)^j\times (\PP^2)^k, \mathcal{O}_{(\PP^1)^j\times (\PP^2)^k}(s_1,\ldots, s_j, t_1, \ldots, t_k))
\]
is secant defective if and only if one of the following holds up to a permutation of $\{s_1, \ldots, s_j\}$
\begin{itemize}
    \item $j=0$, $k=1$ and $t_1\in \{2,4\}$;
    \item $j=0$, $k=2$ and $t_1=t_2 = 2$;
    \item $j=2$, $k=0$ and $s_1=2, s_2 = 2a$ for some $a\in \mathbb{Z}_{>0}$;
    \item $j=3$, $k=0$ and $s_1=2, s_2 = 2, s_3=2$;
    \item $j=1$, $k=1$ and $s_1=2a$, $t_1=2$ for some $a\in \mathbb{Z}_{>0}$.
    \item $j=2$, $k=1$ and $s_1=s_2=t_1=2$.
\end{itemize}
\end{theorem}
\begin{proof}
By \Cref{prod:P2} and \cite{lp} it is enough to consider the cases where $j,k > 0$. Furthermore, the case of $j=k=1$ is known (cf. \cite{BaurDraisma}) so we assume that $j+k\geq 3$.

If $k\geq 3$ or $k=2$ and $(t_1,t_2)\neq (2,2)$, then the corresponding pair is not secant defective. This can be seen by \Cref{prod:P2} and a repeated application of  \Cref{a4.1} ({\cite[Theorem 2]{b}}).

If $k=2$ and $(t_1,t_2)=(2,2)$ then it follows from \Cref{122} that $(\PP^1\times \PP^2\times \PP^2, \mathcal{O}_{\PP^1\times\PP^2\times \PP^2}(t,2,2))$ is not secant defective for every $t\geq 2$.

The case of more factors of $\PP^1$ follows by a repeated application of \Cref{a4.1}.

Finally assume that $j\geq 2$ and $k=1$. If at least one $s_i$ or $t_1$ is larger than $2$ then we conclude that the pair is not secant defective by \Cref{112} and induction on $j$ using \Cref{a4.1}. 

We are left with the case that $j\geq 2$, $k=1$ and $s_1=\cdots = s_j=t_1 = 2$. If $j=2$, then as recalled in \Cref{112}, the pair is secant defective by \cite{cgg}. If $j=3$ the pair is not secant defective by a computer check (with Macaulay2 \cite{M2}). The general case follows from this and a repeated application of \Cref{a4.1}.
\end{proof}

\begin{theorem}\label{thm:mn2}
If $d,e \geq 3$ and $n\geq m\geq 1$, then the pair 
\[
(\PP^m\times\PP^n\times (\PP^2)^k, \Oo_{\PP^m\times\PP^n\times (\PP^2)^k} (d,e,t_1, \ldots, t_k))
\]
is not secant defective for all $k$ and $t_1\geq t_2\geq \cdots \geq 2$.
\end{theorem}
\begin{proof}\hfill
\begin{enumerate}
    \item  Assume first that $m+n\geq 8$. By \cite{go} and \Cref{P1orP2}, it is enough to show that the pair $(\PP^m\times\PP^n, \Oo_{\PP^m\times\PP^n} (3,3))$ satisfies the assumptions of \Cref{a5}. That is we need to check that
    \[
    \binom{n+3}{3}\binom{m+3}{3} \geq \frac{1}{81}(27(m+n)^3+144(m+n)^2+210(m+n)+79).
    \]
    Let $r=m+n$. Observe that $27r^3+144r^2+210r+79 \leq 27r^3+162r^2+324r+216 = (3r+6)^3$. Therefore, the right side is bounded from the above by $\frac{1}{3} (r+2)^3.$
    For $i=1,2,3$, we have $(n+i)((r-n)+i) \geq (1+i)((r-1)+i)$.    
    Therefore,
    \[
    \binom{n+3}{3}\binom{m+3}{3} \geq \frac{4\cdot 3 \cdot 2}{36} (r+2)(r+1)(r).
    \]
    Now it is enough to observe that $2(r+1)r \geq (r+2)^2$. This holds since $r\geq 8$. 
    
    \item Now assume that $m+n \leq 7$ and $m,n\geq 2$. By \cite{go} and \Cref{P1orP2}, it is enough to show that the pair $(\PP^m\times\PP^n, \Oo_{\PP^m\times\PP^n} (3,3))$ satisfies the assumptions of \Cref{a5}. There are finitely many cases to consider. We omit the calculations.
    
    \item If $m=1$ and $4\leq n \leq 6$, then we argue as in (2).

    \item If $m=1$, $n=3$ and $e\geq 4$, then we argue as in (2) with $(\PP^1\times\PP^3, \Oo_{\PP^1\times\PP^3} (3,4))$ instead of $(\PP^m\times\PP^n, \Oo_{\PP^m\times\PP^n} (3,3))$.
    \item Assume that $m=1$, $n=3$, $e=3$ and either $k=0$ or $t_1 \geq 3$. If $k=0$, then the pair is not secant defective by \cite{go}. If $k\geq 1$ and $t_1\geq 3$, then the pair $(\PP^1\times \PP^3\times \PP^2, \Oo_{\PP^1\times \PP^3\times \PP^2}(d,e,t_1))$ is not secant defective by \cite{go} and \Cref{a4.1} ({\cite[Theorem 2]{b}}). Since $h^0(\PP^1\times \PP^3\times \PP^2, \Oo_{\PP^1\times \PP^3\times \PP^2}(d,e,t_1)) \geq 4\cdot 20 \cdot 10 > 231$ we conclude by \Cref{P1orP2}.

    \item Assume that $m=1$, $n=3$, $e=3$, $k\geq 1$ and all $t_i$ are equal to $2$. By \Cref{P1orP2} it is enough to consider the case  $k=1$. The pair $(\PP^3\times \PP^2, \mathcal{O}_{\PP^3\times \PP^2}(3,2))$ is not secant defective by a computer check. Therefore, so is  $(\PP^1\times \PP^3\times \PP^2, \mathcal{O}_{\PP^1\times \PP^3\times \PP^2}(d,3,2))$ by \Cref{a4.1}.

    \item If $m=1$ and $n\in \{1,2\}$ the pair is not secant defective by \Cref{thm:p1andp2}.
\end{enumerate} \vspace{-0.4cm}
\end{proof}

\section{Proofs of Claims}\label{Section:claims}

For the sake of completeness, this section compiles the proofs of the Claims that primarily involve computational steps.

\hrule 

\begin{proof}[Proof of \Cref{claim1}]
Let $y_0$ be the maximal integer such that $y_0 \leq \lfloor \alpha/(r+1) \rfloor$ and $y_0 \equiv 3\alpha (\operatorname{mod} r+2)$. Let $x_0$ be the unique integer such that $(r+2)x_0 + y_0 = 3\alpha$. We show that
\begin{equation}\label{eq1}
    (r+2)(z-x_0-y_0) + y_0 \leq 2\alpha.
\end{equation}
By the definition of $y_0$ we have $y_0 \geq \lfloor \alpha/(r+1) \rfloor-(r+1)$. Therefore, 
\begin{equation}\label{eq2}
    (r+1)y_0 \geq \alpha - r - (r+1)^2.
\end{equation}
Multiplying \cref{eq1} by $(r+3)(r+1)$ and rearranging we need to show that
\begin{multline*}
 (r+1)(r+2)(r+3)z \leq 2\alpha(r+1)(r+3) + (r+3)(r+1)((r+2)x_0 + (r+1)y_0)\\
 = 2\alpha(r+1)(r+3) + (r+3)(r+1)(3\alpha + ry_0)
 =5\alpha (r+1)(r+3) + r(r+1)(r+3)y_0.
\end{multline*}
Using the definition of $z$ and \cref{eq2} it is sufficient to show that
\begin{equation*}
    (r+1)(r+2)(6\alpha + (r+2)) \leq 5\alpha (r+1)(r+3) + r(r+3)(\alpha - (r+(r+1)^2))
\end{equation*}
which is equivalent to
\begin{equation}\label{eqcon1}
5\alpha r + 3\alpha  \geq r^4 + 7r^3 + 15r^2+11r+4.
\end{equation}
Using the lower bounds on $\alpha$ from the statement we see that \cref{eqcon1} holds.

In all cases from the statement, we have  $\alpha \geq 3(r+1)^2+r = 3r^2+7r+3$. It follows from $y_0 \geq \lfloor \alpha/(r+1) \rfloor-(r+1)$ that  $y_0 \geq 2r+2$. Let $y_1$ be the smallest integer $y\geq 2r+2$ such that $y\equiv 3\alpha (\operatorname{mod} r+2)$ and $(r+2)(z-x-y)+ y \leq 2\alpha$ where $x$ is the unique integer for which $(r+2)x+y = 3\alpha$. Let $x_1$ be the unique integer for which $(r+2)x_1 + y_1 = 3\alpha$. 
By construction we have $(r+2)x_1 + y_1 = 3\alpha$, $ 2r+2 \leq y_1 \leq \lfloor \alpha/(r+1) \rfloor$ and $(r+2)(z-x_1-y_1)+y_1 \leq 2\alpha$. We need to show that $x_1+y_1\leq z$ and $z-x_1-y_1 \leq \alpha - (r+1)y_1$.

We start with the latter inequality. We multiply it  by $(r+2)(r+3)$ and rearrange to obtain
\begin{multline*}
(r+2)(r+3)z \leq (r+2)(r+3)\alpha + (r+3)(r+2)x_1 - r(r+2)(r+3)y_1\\
 = (r+2)(r+3)\alpha + 3(r+3)\alpha - (r+1)^2(r+3)y_1.
\end{multline*}
Using the definition of $z$ it is sufficient to show that
\[
(r+2)(6\alpha) + (r+2)^2 \leq (r+5)(r+3)\alpha - (r+1)^2(r+3)y_1
\]
or equivalently,
\begin{equation}\label{eq3}
    \alpha(r^2+2r+3) \geq (r+1)^2(r+3)y_1 + (r+2)^2.
\end{equation}
By definition of $y_1$ we have 
\begin{itemize}
\item either  $y_1 < (2r+2) + (r+2)$, 
\item or $(r+2)(z-(x_1+1)-(y_1-(r+2)) + (y_1 - (r+2)) - 1 \geq 2\alpha$. 
\end{itemize}

In the first case, to obtain \cref{eq3} it is enough to have
\begin{equation}\label{eqcon2}
    \alpha(r^2+2r+3) \geq 3(r+1)^3(r+3)+(r+2)^2 = 3r^4+18r^3+37r^2+34r+13.
\end{equation}
Using the lower bounds of $\alpha$ from the statement we verify that \cref{eqcon2} holds.

In the second case we have $(r+2)z - 5\alpha + (r+2)r-1 \geq ry_1$ and using the fact that $(r+3)z \leq 6\alpha + (r+2)$ we get
\begin{equation}\label{eq5}
    \alpha r - 3\alpha + r^3+6r^2+9r+1 \geq (r+3)r y_1.
\end{equation}
Therefore, to obtain \cref{eq3} it is sufficient if we have
\[
\alpha(r^3+2r^2+3r) \geq (r+1)^2 (\alpha r - 3\alpha + r^3+6r^2+9r+1) + r(r+2)^2
\]
or, equivalently
\begin{equation}\label{eqcon3}
\alpha(3r^2+8r+3) \geq r^5 + 8r^4+23r^3 + 29r^2 + 15r +1.
\end{equation}
It is enough to verify that this inequality holds for the lower bounds of $\alpha$.

Finally, we verify that $x_1+y_1 \leq z$. We have
\[
(r+3)(r+2)(x_1+y_1) = (r+3)(3\alpha + (r+1)y_1) \leq (r+3)(4\alpha)
\]
and 
\[
(r+2)(r+3)z \geq (r+2)6\alpha - (r+2)^2.
\]
Therefore, it is enough to show that
\[
(r+2)6\alpha - (r+2)^2 \geq (r+3)(4\alpha)
\]
or, equivalently, that
\begin{equation}\label{eqcon4}
2\alpha r \geq r^2+4r+4. 
\end{equation}
This holds by the assumed lower bounds on $\alpha$.
\end{proof}

\hrule

\begin{proof}[Proof of \Cref{claim2}]
We have 
$
(r+3)(r+2)(r+1)r(z-x_1-y_1) \geq (r+2)(r+1)r(6\alpha - (r+2)) - (r+3)(r+1)r(3\alpha + (r+1)y_1)
$
and
\[
(r+3)(r+2)(r+1)r \lceil \alpha/(r+1) \rceil \leq (r+3)(r+2)r (\alpha + r).
\]
Thus, it is sufficient to show that 
\[
(r+2)(r+1)r(6\alpha - (r+2)) - (r+3)(r+1)r(3\alpha + (r+1)y_1) \geq (r+3)(r+2)r (\alpha + r) 
\]
or equivalently, that
\begin{equation*}
 \alpha(2r^3+r^2-3r) -2r^4-10r^3-14r^2-4r \geq r(r+1)^2(r+3)y_1.   
\end{equation*}
Recall from the proof of \Cref{claim1} that $y_1\leq 3r+3$ or we have \cref{eq5}.

If  $y_1 \leq 3r+3$, it is sufficient to show that 
\begin{equation}\label{eq8}
    \alpha(2r^3+r^2-3r) -2r^4-10r^3-14r^2-4r \geq 3r(r+1)^3(r+3).
\end{equation}
This may be verified by using the lower bound for $\alpha$.
Assume that \cref{eq5} holds. We show that 
\[
\alpha(2r^3+r^2-3r)-2r^4-10r^3-14r^2-4r \geq (\alpha r - 3\alpha + r^3+6r^2+9r+1)(r+1)^2
\]
or equivalently that
\[
\alpha(r^3+2r^2+2r+3) \geq r^5+10r^4+32r^3+39r^2+15r+1.
\]
Using the lower bounds for $\alpha$ we see that this inequality holds.
\end{proof}

\hrule

\begin{proof}[Proof of \Cref{claim3}]
\[
(r+2)(r+3)(z'-z_1+x_1) = (r+2)(r+3)(z-z_1 - 4\alpha + (r+1)a) + (r+3)(3\alpha - y_1) 
\]
\[
\geq (r+2)(10\alpha) - (r+2)(6\alpha) - 4\alpha(r+2)(r+3) + (r+1)(r+3)(4\alpha - r-1) + (r+3)(3\alpha-y_1)
\]
\[
= \alpha(3r+5) - (r^3+5r^2+7r+3) - (r+3)y_1.
\]
It is enough to show that
\[
\alpha(3r+5)(r+1) \geq (r^3+5r^2+7r+3)(r+1) + (r+3)\alpha
\]
or equivalently
\[
(3r^2+7r+2)\alpha \geq r^4+6r^3+12r^2+10r+3
\]
which is true by the lower bounds on $\alpha$ and $r$.
\end{proof}

\hrule

\begin{proof}[Proof of \Cref{claim11}]
We have 
\[
(r+2)(r+3)(a+b) = (r+3)((t+1)\alpha + (r+1)b) \leq (r+3)(t+1)\alpha + (r+1)^2(r+3)
\]
and 
\[
(r+2)(r+3)z \geq (r+2)\left(\binom{t+2}{2}\alpha - (r+2)\right)
\]
so it is sufficient to show that
\[
(r+2)\left(\binom{t+2}{2}\alpha - (r+2)\right) \geq (r+3)(t+1)\alpha + (r+1)^2(r+3)
\]
or equivalently, that
\[
\alpha(t+1)(tr+2t-2) \geq 2(r+1)^2(r+3)+2(r+2)^2.
\]
It is enough to show the above inequality for $t=3$. Then it can be verified using the lower bounds on $\alpha$ from \Cref{a4}.
\end{proof}

\hrule 

\begin{proof}[Proof of \Cref{claim7}]
    In order to obtain the first part of the claim, we to apply \Cref{a3} \cref{seconda} with $(Y,\mathcal{L}) = (\PP^1\times \PP^1, \mathcal{O}_{\PP^1\times \PP^1}(2a,2))$,  $z=z'$ and $u=\widetilde z$. In particular, $r=2$ and $\alpha = 3(2a+1)$. 
    We first verify that the pair $(Y,\Ll)$ is not-$s$-secant defective for all $s\in \{z', \widetilde{z}, \lfloor 3(2a+1)/4 \rfloor \}$.
    By \cite{lp} only the $(2a+1)$-st secant variety of $(\PP^1\times\PP^1, \mathcal{O}_{\PP^1\times \PP^1}(2a,2))$ is defective. 
    We have $\lfloor 3(2a+1)/4\rfloor \leq 2a$, $\widetilde{z} \leq 2a$ and
    \begin{equation}\label{cl7_eq0}
        z' \geq 2a+2
    \end{equation}
    In order to show \cref{cl7_eq0}, it is sufficient to show that
    $4(18(2a+1) - 4) - 5(4\overline{z} + \widetilde{z}) - 15\widetilde{z} \geq 20(2a+2)$. Or equivalently, that $27(2a+1) \geq 40a + 56+15\widetilde{z}.$ This inequality is true since $\widetilde{z} \leq 3$ and $a\geq 6$.
             
    Therefore, in order to show the first part of the claim it is sufficient to show that the following are true
    \begin{align}
    4 (\lceil 18(2a+1)/5\rceil - \overline z - \widetilde z) & \leq 6(2a+1) - 6\label{cl7_eq1}\\
    4 (\lceil 18(2a+1)/5\rceil - \overline z - \widetilde z) + \widetilde{z} & \leq 6(2a+1)\label{cl7_eq2}\\
    \widetilde{z} & \leq 2a+1 \label{cl7_eq3}\\
    \lceil 18(2a+1)/5\rceil - \overline z - \widetilde z & \leq 3(2a+1) - 3\widetilde z\label{cl7_eq4}.
    \end{align}
    By definition $\widetilde z \leq 3$, hence \cref{cl7_eq3} holds. It also implies that \cref{cl7_eq2} is a consequence of \cref{cl7_eq1}. Furthermore, assuming that \cref{cl7_eq1} holds, in order to obtain \cref{cl7_eq4} it is sufficient to show that $6\widetilde z \leq 3(2a+1) + 3$. This holds since $a\geq 6$ and $\widetilde{z}\leq 3$. Thus we are left with proving \cref{cl7_eq1}. 
     We have $20(\lceil 18(2a+1)/5\rceil - \overline z - \widetilde z) \leq  4(18(2a+1) + 4) - 5(4\overline z + \widetilde z)-15\widetilde{z} = 27(2a+1) + 16-15\widetilde{z}$
    and our goal is to show that this is at most equal to $30(2a+1)-30$. Equivalently, we need to show that $46 \leq 3(2a+1) + 15\widetilde z.$ This holds since $\widetilde{z}\geq 1$ and $a\geq 6$.
\end{proof}

\bibliographystyle{alpha}
\newcommand{\etalchar}[1]{$^{#1}$}

\end{document}